\documentclass[a4paper,12pt]{amsart}
\usepackage[english]{babel}
\usepackage{amsmath,amsthm,amssymb,amsfonts}
\usepackage{multicol}
\usepackage{todonotes}
 \usepackage{enumitem} 
\usepackage[bookmarks=true,hyperindex,pdftex,colorlinks,citecolor=blue]{hyperref}
\hypersetup{colorlinks=true,  linkcolor=blue, citecolor=red, urlcolor=cyan}
\usepackage{cancel}
\usepackage[normalem]{ulem}

\usepackage{xcolor}
\usepackage[export]{adjustbox}
\usepackage{graphicx}
\usepackage{subcaption}
\usepackage{tikz-cd}

\usepackage{tikz}
\usetikzlibrary{mindmap,backgrounds, calc}
\usetikzlibrary{matrix,chains,scopes,positioning,arrows,fit}
\usetikzlibrary{positioning, shapes, patterns}
\usetikzlibrary{automata} 
\usetikzlibrary{shapes.geometric, arrows, arrows.meta}
\usetikzlibrary{positioning,decorations.markings}

\theoremstyle{plain}
\newcounter{maincoro}

\newtheorem{theorem}{Theorem}[section]
\newtheorem{lemma}[theorem]{Lemma}

\newtheorem{proposition}[theorem]{Proposition}
\newtheorem{corollary}[theorem]{Corollary}
\newtheorem{question}[theorem]{Question}
\theoremstyle{definition}
\newcounter{maintheorem}

\newtheorem{definition}[theorem]{Definition}
\newtheorem{example}[theorem]{Example}

\theoremstyle{remark}
\newtheorem{remark}[theorem]{Remark}
\numberwithin{equation}{section}
\newtheorem*{claim*}{Claim}

\newcommand{\R}{\mathbb{R}}
\newcommand{\C}{\mathbb{C}}
\newcommand{\N}{\mathbb{N}}
\newcommand{\K}{\mathbb{K}}

 \makeatletter

\renewcommand{\tocsection}[3]{%
	\indentlabel{\@ifnotempty{#2}{\bfseries\ignorespaces#1 #2\quad}}\bfseries#3}
\renewcommand{\tocsubsection}[3]{%
	\indentlabel{\@ifnotempty{#2}{\ignorespaces#1 #2\quad}}#3}

\def\@tocline#1#2#3#4#5#6#7{\relax
	\ifnum #1>\c@tocdepth 
	\else
	\par \addpenalty\@secpenalty\addvspace{#2}%
	\begingroup \hyphenpenalty\@M
	\@ifempty{#4}{%
		\@tempdima\csname r@tocindent\number#1\endcsname\relax
	}{%
		\@tempdima#4\relax
	}%
	\parindent\z@ \leftskip#3\relax \advance\leftskip\@tempdima\relax
	\rightskip\@pnumwidth plus1em \parfillskip-\@pnumwidth
	#5\leavevmode\hskip-\@tempdima{#6}\nobreak
	\leaders\hbox{$\m@th\mkern \@dotsep mu\hbox{.}\mkern \@dotsep mu$}\hfill
	\nobreak
	\hbox to\@pnumwidth{\@tocpagenum{\ifnum#1=1\bfseries\fi#7}}\par
	\nobreak
	\endgroup
	\fi}
\AtBeginDocument{%
	\expandafter\renewcommand\csname r@tocindent0\endcsname{0pt}
}
\def\l@subsection{\@tocline{2}{0pt}{2.5pc}{5pc}{}}
\makeatother

\DeclareMathOperator{\co}{co}

\DeclareMathOperator{\id}{Id}

\DeclareMathOperator{\Span}{span}

\newcommand{\nn}[1]{{\left\vert\kern-0.25ex\left\vert\kern-0.25ex\left\vert #1 
		\right\vert\kern-0.25ex\right\vert\kern-0.25ex\right\vert}}

\renewcommand{\geq}{\geqslant}
\renewcommand{\leq}{\leqslant}

\newcommand{\norm}[1]{\left\Vert#1\right\Vert}

\newcommand{\NA}{\operatorname{NA}}

\newcommand{\call}{\mathcal{L}}

\newcommand{\spann}{\operatorname{span}}

\newcommand{\pten}{\ensuremath{\widehat{\otimes}_\pi}}
\newcommand{\iten}{\ensuremath{\widehat{\otimes}_\varepsilon}}
\newcommand{\eps}{\varepsilon}

\newcommand{\INA}{\operatorname{INA}}
\newcommand{\FNA}{\operatorname{FNA}}
 
\newcommand{\wconv}{\stackrel{w}{\to}} 
\newcommand{\set}[1]{\left\{{#1}\right\}} 
\newcommand{\duality}[1]{\left<{#1}\right>} 
\newcommand{\abs}[1]{\left|{#1}\right|} 
\newcommand{\pare}[1]{\left({#1}\right)} 
\newcommand{\wstarconv}{\stackrel{w^*}{\to}} 

\thanks{}

\subjclass[2020]{}

\date{\today}


\begin{document}

\title[Integral projective norm-attaining tensors]{Integral representations of projective norm-attaining tensors}

\author[R.~J.~Aliaga]{Ramón J. Aliaga}
\address[R.~J.~Aliaga]{Instituto Universitario de Matemática Pura y Aplicada, Universitat Politècnica de València, Camino de Vera S/N, 46022, Valencia, Spain \newline
\href{https://orcid.org/0000-0002-2513-7711}{ORCID: \texttt{0000-0002-2513-7711}}}
\email{\texttt{raalva@upv.es}}
\urladdr{https://raalva.wordpress.com/}

\author[S.~Dantas]{Sheldon Dantas}
\address[S.~Dantas]{Czech Technical University in Prague, FEE, Department of Mathematics, Technická 2, 16627, Prague 6, Czech Republic \newline
\href{https://orcid.org/0000-0001-8117-3760}{ORCID: \texttt{0000-0001-8117-3760}}}
\email{\texttt{sheldon.dantas@fel.cvut.cz}}
\urladdr{www.sheldondantas.com}

\author[J.~Guerrero-Viu]{Juan Guerrero-Viu}
        \address[J.~Guerrero-Viu]{Departamento de Matemáticas, Universidad de Zaragoza, C/ Pedro Cerbuna 12, 50009, Zaragoza, Spain \newline
\href{https://orcid.org/0009-0001-2125-5120}{ORCID: \texttt{0009-0001-2125-5120}}}
\email{j.guerrero@unizar.es}

\author[M.~Jung]{Mingu Jung}
        \address[M.~Jung]{Department of Mathematics \& Research Institute for Natural Sciences, Hanyang University, 04763 Seoul, Republic of Korea \newline
\href{https://orcid.org/0000-0003-2240-2855}{ORCID: \texttt{0000-0003-2240-2855}}}
\email{mingujung@hanyang.ac.kr}
\urladdr{https://sites.google.com/view/mingujung/}

\author[Ó.~Roldán]{Óscar Roldán}
\address[Ó.~Roldán]{Departamento de Análisis Matemático, Universidad de Valencia, Avenida Vicente Andrés Estellés 19, 46100, Burjasot (Valencia), Spain \newline
\href{https://orcid.org/0000-0002-1966-1330}{ORCID: \texttt{0000-0002-1966-1330}}}
\email{oscar.roldan@uv.es}

\thanks{}

\date{\today}

\begin{abstract} 
We introduce a Bochner integral approach to projective norm attainment in tensor products of Banach spaces by defining the class of integral projective norm-attaining tensors. This framework provides a broader, measure-theoretic approach to the study of projective norm attainment in tensor products of Banach spaces. We show that every integral norm-attaining tensor can be approximated in norm by norm-attaining tensors with finite representations. As a consequence, the Bishop-Phelps type density problem for classical norm-attaining tensors is equivalent to the corresponding density problem for integral norm-attaining tensors. Moreover, we prove that if an integral projective norm-attaining tensor represented by a Radon measure is an extreme point, then it must be an elementary tensor.
We further investigate weaker topological versions of integral norm-attainment, including weak and weak$^*$ integral representations, providing sufficient conditions for the existence of Bochner representations. Finally, we extend known constructions of projective tensor products containing non-norm-attaining tensors to the integral setting. We show, for instance, that $L_1\widehat{\otimes}_\pi L_p$ and the real $c_0\widehat{\otimes}_\pi L_p$ contain non-norm-attaining tensors for $1<p<\infty$.

\end{abstract}

\maketitle

\section{Introduction}

The study of norm-attaining mappings has played a central role in Banach space theory since the classical results of reflexivity due to James and the Bishop-Phelps theorem on the denseness of norm-attaining functionals. These results revealed that norm attainment is closely connected to the geometry of the unit ball of the involved Banach spaces and have motivated a vast literature extending the problem to bounded linear operators (see, for instance, \cite{B, H, S, U}), bilinear mappings (see, for instance, \cite{gasp, AFW, Choi}), polynomials (see, for instance, \cite{AGM, ADM}), and other nonlinear settings (see \cite{CCGMR, DM, Godefroy15,kms}).

Projective tensor products, on the other hand, provide a natural framework in which several of these phenomena meet together (see \cite{DGJR, DJRR, GGR, Rueda}). Given Banach spaces $X$ and $Y$, the projective tensor product $X \pten Y$ is canonically related to bilinear mappings and operators through natural identifications. While classical norm attainment concerns elements of the dual space, a more recent line of research has focused on norm attainment at the level of the tensor product itself by trying to find ``good enough'' representations in the tensor product. Inspired by the theory of nuclear operators (see \cite[Definition 2.1]{DJRR}), an element $z\in X\pten Y$ is said to \emph{attain its projective norm} if the infimum in the definition of its projective norm is actually a minimum, that is, if $z$ admits an \emph{optimal representation} of the form $z=\sum_{n=1}^\infty x_n\otimes y_n$ with $\|z\|_\pi=\sum_{n=1}^\infty \|x_n\|\,\|y_n\|$. The set of all such tensors is denoted by $\mathrm{NA}_\pi(X \pten Y)$.

The first systematic results in this context \cite{DGJR, DJRR} showed that projective norm attainment holds in a number of classical situations. In particular,  $\mathrm{NA}_\pi(X \pten Y)=X \pten Y$ whenever $X$ and $Y$ are finite-dimensional \cite[Proposition 3.5]{DJRR}, when $X=\ell_1(I)$ and $Y$ is arbitrary \cite[Proposition 3.3]{DJRR}, when $X=Y$ is a Hilbert space (see \cite[Proposition 3.8]{DJRR} for the complex case and \cite[Theorem 2.2]{GGM} for the real case), or in some settings involving polyhedral spaces \cite[Theorem 4.1]{DGJR}. On the other hand, much effort has been devoted to studying the denseness of this subset and to obtaining Bishop-Phelps type results in this context. It is now known that $\mathrm{NA}_\pi(X \pten Y)$ is dense under a variety of geometric and approximation hypotheses, such as the metric $\pi$-property, uniform convexity, or the Radon--Nikod\'ym and approximation properties in suitable dual spaces. These results certainly cover a wide class of Banach spaces including, from our point of view, most of the classical ones.

Nevertheless, projective norm attainment presents some similar behavior to what happens to the classical Bishop-Phelps theorem when one tries to extend it to the operator case: there exist Banach spaces $X$ and $Y$ for which $\mathrm{NA}_\pi(X \pten Y)$ fails to be dense (see \cite[Section 5]{DJRR}). In fact, there are even examples in which finite sums of elementary tensors do not attain their projective norm (see \cite[Proposition 5.5]{DJRR}). Moreover, it may happen that both $\mathrm{NA}_\pi(X \pten Y)$ and its complement are dense in $X \pten Y$ (see \cite[Remark 3.2.(2) and Theorem 3.3]{Rueda}), showing that norm-attaining and non-norm-attaining elements can coexist in nontrivial scenarios. This shows that projective norm attainment is strongly dependent on the particular underlying spaces that one is working with. In other words, the general picture that projective norm attainment gives us is that, depending on the geometric and approximation properties of $X$ and $Y$, the subset $\mathrm{NA}_\pi(X \pten Y)$ may coincide with the whole space, be dense, fail to be dense, or even have dense complement. Despite the significant progress achieved in recent years (we send the reader to \cite{GGR} for the latest advances and an overview of the current state of research), a complete understanding of the structure and size of the set of projective norm-attaining tensors is still lacking.

The aim of the present paper is to further contribute to this line of research. To this end, we introduce the notion of \emph{integral projective norm-attaining tensors} (see Definition~\ref{def:INA}), which extends the classical concept of projective norm attainment by allowing continuous Bochner integral representations instead of discrete tensor decompositions. This Bochner integral approach provides a broader (at least formally) framework for the study of projective norm attainment. Within this setting, we investigate when $\NA_{\pi}$ coincides with the whole space and when it is norm dense.

Our approach is inspired by an analogous development in the context of Lipschitz-free spaces $\mathcal{F}(M)$. In these spaces, one similarly has a set of canonical generating elements called {\it (elementary) molecules}, which play a role similar to elementary tensors in $X\pten Y$: every element $\mu\in\mathcal{F}(M)$ can be expressed as an absolutely convergent series of molecules $\mu=\sum_{n=1}^\infty m_n$, and $\norm{\mu}$ equals the infimum of $\sum_{n=1}^\infty\norm{m_n}$ over all such representations. When this infimum is attained, $\mu$ is called a {\it convex series of molecules} - a concept analogous to norm-attaining tensors. Replacing a discrete sum with a continuous one (i.e., an integral), one obtains {\it convex integrals of molecules}, which lead to an elegant theory that yields solutions to previously open questions on Lipschitz-free spaces \cite{APS24}. One could hope that an analogous concept in projective tensor products would similarly advance our knowledge.

The paper is organized as follows.
First, in Section \ref{section:ina}, we establish several basic properties of integral projective norm-attaining tensors, including stability under passing to complemented subspaces and characterizations involving norm-attaining bilinear forms. Our first main result shows that every integral norm-attaining tensor can be approximated in norm by (finitely) norm-attaining tensors. In particular, this yields that the Bishop-Phelps type density problem for classical norm-attaining tensors is equivalent to the corresponding density problem for integral norm-attaining tensors. In addition, we prove that if an integral projective norm-attaining tensor witnessed by a Radon measure is an extreme point of the unit ball, then it is necessarily an elementary tensor, highlighting a strong connection between integral projective norm attainment and the extremal structure of the projective tensor product. Next, in Section \ref{section:general-ina}, we study weaker topological versions of integral norm-attainment, including weak and weak$^*$ integral representations. Namely, we provide sufficient conditions under which every projective tensor element is integral norm-attaining in this weaker topological sense and obtain density results for classical projective norm-attaining tensors. Some examples are presented showing that vector-valued integral representations may fail to exist even when scalar integral representations are available, illustrating the subtle role of Bochner integrability in this context. Finally, in Section \ref{section:non-norm-attaining} we extend known examples of projective tensor products containing non-norm-attaining tensors to the integral setting. In particular, we obtain new results for tensor products involving $c_0$-type spaces and $L_1$-spaces, including results that are new even in the classical projective tensor setting.
Namely, we show that the following infinite-dimensional projective tensor products contain non-norm-attaining tensors:
\begin{itemize}
\item the real $c_0\pten Y$, where $Y$ and $Y^*$ are strictly convex (Theorem \ref{thm:c0Y}),
\item $c_0\pten L_p(\mu)$ for $p=\frac{2n}{2n-1}$, $n\in\N$ (see Theorem \ref{theo:ellps-2} and Corollary \ref{theo:ellps} for a more general statement),
\item $L_1(\mu)\pten Y$, where $\mu$ is not purely atomic, $Y$ contains a two-dimensional strictly convex subspace, and $Y^*$ has the Radon-Nikod\'ym property (see Theorem \ref{theorem:L1-result2} for a more general statement).
\end{itemize}

\subsection{Notation and preliminaries}

In this section we introduce the notation and collect the preliminary material that will be used throughout the paper.

The letters $X$, $Y$ and $Z$ denote Banach spaces over the scalar field $\K$, where $\K = \R$ or $\C$. The closed unit ball and the unit sphere of a Banach space $X$ are denoted by $B_X$ and $S_X$, respectively. For a subset $B\subseteq X$, we write $\co(B)$ for its convex hull. Given a convex subset $C$ of Banach space $X$, a point $x\in C$ is called an extreme point of $C$ if it cannot be expressed as a nontrivial convex combination of other points of $C$.

Given Banach spaces $X$ and $Y$, we denote by $\mathcal{L}(X,Y)$ the Banach space of all bounded linear operators from $X$ into $Y$. By $\mathcal{B}(X\times Y)$ we mean the Banach space of all continuous bilinear forms on $X\times Y$ with values in $\mathbb{K}$. We say that $T \in \mathcal{L}(X,Y)$ is {\it norm-attaining} if there exists $x_0 \in S_X$ such that $\|T(x_0)\| = \|T\|$. Analogously,
$B\in\mathcal{B}(X\times Y)$ is norm-attaining if there exist $x_0\in S_X$, $y_0\in S_Y$ such that $\abs{B(x_0,y_0)}=\norm{B}$.
We denote by $\NA(X, Y)$ the linear operators which attain the norms and by $\NA(X \times Y)$ the subset of $\mathcal{B}(X \times Y)$ of all norm-attaining bilinear forms.

For a comprehensive treatment of tensor products and nuclear operators we refer to \cite{deflo, Ryan}.
The \emph{projective} and \emph{injective tensor products} of $X$ and $Y$, denoted respectively
by $X\pten Y$ and $X\iten Y$, are defined as the completion of the algebraic tensor product $X\otimes Y$ endowed with the norms
\begin{equation} \label{projective-norm}
\|z\|_{\pi}
:= \inf \left\{ \sum_{i=1}^n \|x_i\|\,\|y_i\| : z=\sum_{i=1}^n x_i\otimes y_i \right\},
\end{equation}
where the infimum is taken over all such representations of $z$, and
\begin{equation*}
\Big\|\sum_{i=1}^n x_i\otimes y_i\Big\|_{\varepsilon}
= \sup \left\{ \left| \sum_{i=1}^n x^*(x_i)y^*(y_i) \right| :
x^*\in B_{X^*},\ y^*\in B_{Y^*} \right\}.
\end{equation*}
It is well known that 
\begin{itemize}
\item[(a)] $B_{X\pten Y}=\overline{\co}(\{x\otimes y:\, x\in B_X,\, y\in B_Y\})$ and 
\item[(b)] $(X\pten Y)^*=\mathcal{B}(X\times Y)=\mathcal{L}(X,Y^*)$. 
\end{itemize}

Recall that a Banach space $X$ is said to have the \emph{approximation property} (AP, for short) if, for every compact subset $K\subseteq X$ and every $\varepsilon>0$, there
exists a finite-rank operator $T\colon X\to X$ such that $\|T(x)-x\|\leq \varepsilon$ for all $x\in K$. If $T$ can be always chosen to satisfy $\norm{T}\leq\lambda$ for some $\lambda<\infty$, then $X$ is said to have the \emph{bounded approximation property} (BAP, for short).

A Banach space $X$ is said to have the \emph{Radon-Nikod\'ym property} (RNP, for short) if the classical Radon-Nikod\'ym theorem holds for $X$. That is, if $\Sigma$ is a $\sigma$-algebra of subsets of a set $\Omega$ and $\mu:\Sigma\to X$ is a vector measure of bounded variation that is absolutely continuous with respect to a finite positive measure $\lambda$, then there exists a $\lambda$-Bochner integrable function $f: \Omega\to X$ such that 
\[
\mu(E) = \int_E f \, d\lambda
\]
for every $E \in \Sigma$. Many classes of Banach spaces, including reflexive spaces and separable dual spaces, are known to possess the RNP. Together with the BAP, the RNP plays an important role in the theory of tensor products of Banach spaces, particularly in the duality between injective and projective tensor product spaces.

We will also need some notions of measure theory and vector-valued integration; we recommend \cite{Bogachev} and \cite{DU} for reference. By a measure $\mu$ on a Hausdorff space $\mathcal{X}$ we always mean a scalar-valued countably additive measure, and we assume that it is Borel (i.e., it is defined on all Borel sets of $\mathcal{X}$). We say that $\mu$ is {\it Radon} if it is regular, tight, and finite (i.e., it has finite total variation). If $\mu$ is positive, then this is equivalent to $\mu(\mathcal{X})<\infty$ and
$$
\mu(E) = \sup\set{\mu(K) \,:\, K\subset E\text{ compact}} = \inf\set{\mu(U) \,:\, U\supset E\text{ open}}
$$
for all Borel sets $E\subset\mathcal{X}$. Given two Borel measures $\mu$ and $\lambda$ on $\mathcal{X}$, with $\mu$ positive, we write $\lambda\ll\mu$ if $\lambda$ is absolutely continuous with respect to $\mu$, that is, $\lambda(E)=0$ for every Borel set $E\subset\mathcal{X}$ for which $\mu(E)=0$. Finally, we denote by $|\mu|$ the \emph{total variation} of $\mu$, and we say $\mu$ is \emph{concentrated on} a Borel set $A$ if $\mu(E)=\mu(E\cap A)$ for every Borel subset $E$ of $\mathcal{X}$.

Given a Borel measure $\mu$ on $\mathcal{X}$, a function $f:\mathcal{X}\to X$ taking values in a Banach space $X$ is {\it $\mu$-measurable} if and only if it is
\begin{itemize}
\item {\it weakly measurable}, i.e., $x^*\circ f:\mathcal{X}\to\R$ is Borel for every $x^*\in X^*$, and
\item {\it $\mu$-essentially separably valued}, i.e., there is a $\mu$-null Borel set $N\subset\mathcal{X}$ such that $\set{f(x):x\notin N}$ is contained in a separable subspace of $X$.
\end{itemize}
If $f$ is $\mu$-measurable and moreover satisfies $\int_\mathcal{X}\norm{f}\,d\abs{\mu}<\infty$, then $f$ is said to be {\it $\mu$-Bochner integrable}. The Bochner integral $\int_\mathcal{X}f\,d\mu$ is then a well-defined element of $X$ that can be obtained as the limit of integrals of simple functions approximating $f$ in the usual way. If so, then for any $T\in\mathcal{L}(X,Y)$ the map $Tf=T\circ f:\mathcal{X}\to Y$ is also $\mu$-Bochner integrable and the equality $\int_\mathcal{X}Tf\,d\mu=T\pare{\int_\mathcal{X}f\,d\mu}$ holds true.

If $f:\mathcal{X}\to\mathcal{Y}$ is a Borel map between Hausdorff spaces $\mathcal{X}$ and $\mathcal{Y}$, and $\mu$ is a Borel measure on $\mathcal{X}$, then the {\it pushforward measure} $f_\sharp\mu$ is a Borel measure on $\mathcal{Y}$ defined by $f_\sharp\mu(E)=\mu(f^{-1}(E))$ for Borel $E\subset\mathcal{Y}$. If $\mu$ is positive, then so is $f_\sharp\mu$ and moreover $\norm{f_\sharp\mu}=\norm{\mu}$.
Also, if $\mu$ is Radon and $f$ is continuous, then $f_\sharp\mu$ is Radon as well.
Pushforward measures satisfy the following ``change-of-variables'' property: if $g:\mathcal{Y}\to X$ is weakly measurable and such that $g\circ f$ is $\mu$-Bochner integrable, then $g$ is $f_\sharp\mu$-Bochner integrable and $\int_\mathcal{Y}g\,d(f_\sharp\mu) = \int_\mathcal{X}(g\circ f)\,d\mu$.

\section{Integral projective norm-attaining tensors} \label{section:ina}

In this section, we introduce the notion of integral norm-attaining tensors, which can be viewed as a natural extension of projective norm attainment obtained by replacing discrete tensor representations as in (\ref{eqNA}) below with continuous integral representations in $X\pten Y$ as in (\ref{eqINA}).

\subsection{Definition} The following definition is equivalent to the one mentioned in the introduction: given two Banach spaces $X$ and $Y$, we say that an element $u \in X \pten Y$ {\it attains its projective norm} or it is a {\it norm-attaining tensor} (see \cite{DJRR}) if there are a sequence $(x_n, y_n) \subseteq B_X \times B_Y$ and $\lambda_n \geq 0$ such that 
\begin{equation} \label{eqNA}
    u = \sum_{n=1}^{\infty} \lambda_n x_n \otimes y_n \ \ \ \ \ \mbox{and} \ \ \ \ \ \sum_{n=1}^{\infty} \lambda_n = \|u\|_{\pi} .
\end{equation}
We denote by $\NA_{\pi}(X \pten Y)$ the subset of $X \pten Y$ of such elements. 

We now introduce a ``continuous version'' of projective norm attainment as follows.

\begin{definition} \label{def:INA} Let $X, Y$ be Banach spaces. We say that $u\in X\pten Y$ is an {\it integral projective norm-attaining tensor} if there exists a finite positive Borel measure $\mu$ on $B_X \times B_Y$ such that the mapping
$$
\varphi:B_X\times B_Y\to X\pten Y \ \ \ \mbox{defined by} \ \ \ \varphi(x,y)=x\otimes y
$$
is $\mu$-Bochner integrable, and $u$ can be represented by the Bochner integral
\begin{equation} \label{eqINA}
 u = \int_{B_X \times B_Y} x \otimes y \ d \mu(x,y) \ \ \ \ \ \mbox{and} \ \ \ \ \ \|\mu\| = \|u\|_{\pi}.
\end{equation}
In that case, we say that $u$ is {\it witnessed} or {\it represented by $\mu$}.
\end{definition}

We will write $\INA_{\pi}(X \pten Y)$ to denote the subset of $X \pten Y$ of all integral projective norm-attaining tensors. Let us observe that projective norm-attaining tensors arise as a particular case of Definition \ref{def:INA}. Indeed, if $u \in X \pten Y$ can be written as in (\ref{eqNA}), then (\ref{eqINA}) corresponds to choosing the discrete measure
\[
\mu=\sum_{n=1}^{\infty} \lambda_n\, \delta_{(x_n,y_n)} 
\]
where $\delta_{(x_n,y_n)}$ denotes the Dirac measure concentrated at the element $(x_n,y_n)\in B_X\times B_Y$. In this sense, integral projective norm-attaining tensors extend formally the classical notion by allowing continuous measure representations instead of purely discrete ones. In other words, for all Banach spaces $X$ and $Y$, we have that 
\[
    \NA_{\pi}(X \pten Y) \subseteq \INA_{\pi}(X \pten Y) \subseteq X \pten Y.
\]
In particular, whenever $\NA_{\pi}(X \pten Y) = X \pten Y$, we have that $\INA_{\pi}(X \pten Y) = X \pten Y$ and, by using the same inclusions, whenever $\overline{\NA_{\pi}(X \pten Y)} = X \pten Y$ we have that $\overline{\INA_{\pi}(X \pten Y)} = X \pten Y$. In fact, as we will prove in a moment (see Theorem \ref{density1} below), it turns out that 
\[
\INA_{\pi}(X \pten Y) \subseteq \overline{\NA_{\pi}(X \pten Y)} 
\]
for all Banach spaces $X,Y$. This shows that the problem of determining whether $\NA_{\pi}$ is dense in $X \pten Y$ is equivalent to the corresponding problem for the set $\INA_{\pi}$ (we send the reader to Corollary \ref{INA-NA-are-dense-equivalence}).

Of course, this definition is only meaningful if it does not already agree with the discrete version of norm-attaining tensors. So the following question is natural.

\begin{question}\label{q:na=ina}
Is $\INA_{\pi}(X \pten Y) = \NA_{\pi}(X \pten Y)$ for all Banach spaces $X,Y$?
\end{question}

\noindent Unfortunately, we have not been able to answer this question. We strongly suspect that the answer is negative, but either answer (positive or negative) would be very interesting.

\begin{remark} \label{remark:INA-properties}
Concerning Definition~\ref{def:INA}, we record the following observations.

\begin{itemize}

\item[(a)]
It is no loss of generality to restrict ourselves to \emph{positive} measures in \eqref{eqINA}. Indeed, let $\mu$ be a scalar-valued Borel measure on $B_X\times B_Y$ satisfying \eqref{eqINA}. By the Radon-Nikodým theorem, there exists a Borel function $h\in L_1(|\mu|)$ such that $d\mu = h\, d|\mu|$ and $|h|=1$ $|\mu|$-almost everywhere. Define the Borel mapping $f\colon B_X\times B_Y \to B_X\times B_Y$ by $f(x,y) = (h(x,y)\,x,\,y)$. Then, since $\varphi$ is continuous, we get that
\begin{align*}
u &= \int_{B_X\times B_Y} \varphi(x,y)\, h(x,y)\, d|\mu|(x,y) \\
&= \int_{B_X\times B_Y} \varphi\circ f (x,y)\, d|\mu|(x,y) \\
&= \int_{B_X\times B_Y} \varphi(x,y)\, d(f_{\sharp}|\mu|)(x,y)
\end{align*}
using change of variables. Consequently, $u$ admits a representation of the form~\eqref{eqINA} with the positive measure $f_{\sharp}|\mu|$ in place of $\mu$.

\vspace{0.2cm}
\item[(b)] We may replace ``$B_X\times B_Y$'' with ``$S_X\times S_Y$'' in Definition \ref{def:INA} as the representing positive measure $\mu$ in \eqref{eqINA}$\,$ must be concentrated on $S_X\times S_Y$. Indeed, assume on the contrary that $\mu$ assigns positive measure to $B_X\times B_Y\setminus S_X\times S_Y$. Then, by countable additivity, there exists $\alpha<1$ such that the Borel set
\[
\Omega = (B_X\times \alpha B_Y) \cup (\alpha B_X \times B_Y)
\]
satisfies $\mu(\Omega)>0$. It follows that
\begin{align*} 
\|u\|_\pi &\leq \int_{B_X\times B_Y} \|x\otimes y\|_\pi  \, d\mu(x,y) \\
&\leq \alpha\,\mu(\Omega) + \mu\left(B_X\times B_Y\setminus \Omega\right) \\
&< \|\mu\| 
\end{align*} 
which contradicts the equality $\|\mu\|=\|u\|_\pi$ in \eqref{eqINA}.

\vspace{0.2cm}
\item[(c)] Bochner integrability of $\varphi$ is a more subtle issue. Since $\varphi$ is clearly continuous and
\[
\int_{B_X\times B_Y} \|x\otimes y\|_\pi \, d\mu(x,y)
\leq \int_{B_X\times B_Y} d\mu(x,y)
= \|\mu\| < \infty ,
\]
the only unclear condition is whether $\varphi$ is $\mu$-essentially separably valued. This holds when $\mu$ is Radon, as $\mu$ is then concentrated on a $\sigma$-compact set $C\subset B_X\times B_Y$ by regularity, so $\varphi(C)\subset X\pten Y$ is again $\sigma$-compact, hence separable. It is known that every finite Borel measure on a complete metric space $M$ is Radon if (and only if) the density character of $M$ is smaller than the least real-valued measurable cardinal \cite[Proposition 7.2.10]{Bogachev}. Thus, Bochner integrability is guaranteed when $X$ and $Y$ have density less than the smallest measurable cardinal. This is a purely axiomatic issue, as the existence of measurable cardinals is independent of ZFC.

\end{itemize}
\end{remark}

\subsection{Basic properties} We will now describe some general properties of integral norm-attaining tensors. Some of them are simply integral counterparts of properties of norm-attaining tensors established in \cite{DJRR,GGR} and have a similar proof.

We start with the following characterization of elements of $\INA_{\pi}(X\pten Y)$. This is the counterpart of \cite[Theorem 3.1]{DJRR} and the proof follows similar arguments. We present it for the sake of completeness.

\begin{proposition} \label{theo:1-INA} Let $X, Y$ be Banach spaces, and let $\mu$ be a positive Borel measure on $B_X\times B_Y$ such that the Bochner integral
\begin{equation*}
    u = \int_{B_X \times B_Y} x \otimes y \ d\mu(x,y)
\end{equation*}
is well-defined. The following statements are equivalent.
\begin{itemize}
\item[(1)] $u \in \INA_{\pi}(X \pten Y)$, witnessed by $\mu$.
\item[(2)] There is $B \in S_{\mathcal{B}(X,Y)}$ such that $B(x,y) = 1$ for $\mu$-almost all $(x,y)\in B_X\times B_Y$.
\item[(3)] For every $B\in S_{\mathcal{B}(X,Y)}$ with $B(u) = \|u\|_{\pi}$, we have that $B(x,y) = 1$ for $\mu$-almost all $(x,y)\in B_X\times B_Y$.
\end{itemize}
\end{proposition}

\begin{proof} Let us prove that (1) $\Rightarrow$ (3). Let $B$ as in (3). Then
$$
\|u\|_{\pi} = B(u) = \int_{B_X \times B_Y} B(x \otimes y) \ d \mu(x,y) = \int_{B_X \times B_Y} B(x,y) \ d \mu(x,y) .
$$
Taking absolute values, we get
$$
\|u\|_{\pi} \leq \int_{B_X \times B_Y} |B(x,y)| \ d \mu(x,y) \leq \int_{B_X \times B_Y} \|B\| \|x\| \|y\| \, d \mu(x,y) \leq \| \mu\| = \|u\|_{\pi}.
$$
Therefore, $B(x,y) = |B(x,y)| = 1$ $\mu$-almost everywhere. The implication (3) $\Rightarrow$ (2) is obvious. Now, let us prove that (2) $\Rightarrow$ (1). Let $B \in S_{\mathcal{B}(X,Y)}$ such that $B(x,y) = 1$ for $\mu$-almost all $(x,y)$. Then, we have that 
\begin{equation*}
    \|u\|_{\pi} \geq |B(u)| = \left| \int_{B_X \times B_Y} B(x,y) \ d \mu(x,y) \right| = \left| \int_{B_X \times B_Y} d \mu(x,y) \right| = \|\mu\|. 
\end{equation*}
On the other hand, notice that we always have that 
\begin{equation*}
    \|u\|_{\pi} \leq \int_{B_X \times B_Y} \|x\otimes y\|_\pi  \ d \mu(x,y) \leq \int_{B_X \times B_Y} d \mu(x,y) = \|\mu\|. 
\end{equation*}
This means that $u \in \INA_{\pi}(X \pten Y)$, witnessed by $\mu$.
\end{proof}

 We have the following straightforward consequence of Proposition \ref{theo:1-INA}.

\begin{corollary} \label{rem:ina_ac} Let $X, Y$ be Banach spaces and let $u \in \INA_{\pi}(X \pten Y)$ witnessed by a positive Borel measure $\mu$ as in \eqref{eqINA}. If $\mu'$ is another finite positive Borel measure on $B_X\times B_Y$ such that $\mu'\ll\mu$, then the tensor
\[
u' = \int_{B_X\times B_Y} x\otimes y\,d\mu'(x,y)
\]
is well-defined and $u' \in \INA_{\pi}(X \pten Y)$, witnessed by $\mu'$.
\end{corollary}

\begin{proof} Clearly, if $\varphi(x,y)=x\otimes y$ is $\mu$-essentially separably valued then it is also $\mu'$-essentially separably valued, so $u'$ is a well-defined Bochner integral, since $\mu'$ is finite. Moreover, by (1)$\Rightarrow$(2) in Proposition \ref{theo:1-INA}, we have that $\mu$ is concentrated on a set $C\subset B_X\times B_Y$ and there exists $B\in S_{\mathcal{B}(X,Y)}$ such that $B(x,y)=1$ for all $(x,y)\in C$. Since $\mu'\ll\mu$, we have that $\mu'$ is also concentrated on $C$, thus by (2)$\Rightarrow$(1) (by taking the same $B$) we get $u'\in\INA_\pi(X\pten Y)$ witnessed by $\mu'$.
\end{proof}

\begin{corollary} Let $X, Y$ be Banach spaces. If $u\in\NA_\pi(X\pten Y)$ is witnessed by the series $u = \sum_n a_nx_n\otimes y_n$, then any $u'=\sum_n a'_nx_n\otimes y_n$ with $a'_n\in [0,a_n]$ is also a norm-attaining tensor.
\end{corollary}

Another direct consequence of Proposition \ref{theo:1-INA} (see the proof for the implication (1) $\implies$ (3)) is the following. 

\begin{corollary} \label{corollary1:INA} 
    Let $X, Y$ be Banach spaces. If $B \in \mathcal{B}(X \times Y) = (X \pten Y)^*$ attains its norm (as a functional) at an element of $\INA_{\pi}(X \pten Y)$, then $B \in \NA(X \times Y)$.
\end{corollary}

Next, we present integral counterparts of \cite[Lemma 3.9 and Proposition 3.10]{DJRR}. In \cite{DJRR}, these are used to show that there are tensor products where not every projective tensor is norm-attaining. Here, we obtain the (formally) stronger conclusion that not every tensor is integral norm-attaining. In Section \ref{section:non-norm-attaining} below we will exhibit more examples for which $\INA_{\pi}(X \pten Y) \not= X \pten Y$.

\begin{proposition} \label{prop1:INA}Let $X, Y$ be Banach spaces. If $\INA_{\pi}(X \pten Y) = X \pten Y$ then 
\begin{equation*}
    \overline{\NA(X \times Y)}^{\|\cdot\|} = \mathcal{B}(X \times Y).
\end{equation*}
In particular, there are Banach spaces $X$ and $Y$ such that $\INA_{\pi}(X \pten Y) \not= X \pten Y$.
\end{proposition}

\begin{proof}
    By the Bishop-Phelps theorem, the set of all $B \in \mathcal{B}(X \times Y) = (X \pten Y)^*$ attaining their norm as a functional is norm-dense. If every tensor in $X\pten Y$ is integral norm-attaining, then every such $B$ also attains its norm as a bilinear form by Corollary \ref{corollary1:INA}. This proves the first claim. Now, since there  exist Banach spaces $X$ and $Y$ for which the set of norm-attaining bilinear forms $\NA(X\times Y)$ fails to be dense in $\mathcal{B}(X\times Y)$ (see, for instance, \cite[Example~3.12]{DJRR} for a list of examples), this yields Banach spaces $X$ and $Y$ where $\INA_{\pi}(X \pten Y) \not= X \pten Y$. 
\end{proof}

Next, we will extend results from \cite{GGR} concerning the behavior of $\INA_\pi$ when passing to subspaces.
Given operators $T \in \mathcal{L}(X,Z)$ and $S \in \mathcal{L}(Y,W)$, we define $T\otimes S \colon X\pten Y \to Z\pten W$ by $(T\otimes S)(x\otimes y) := T(x)\otimes S(y)$ for $x\in X,\ y\in Y$. This operator extends linearly and continuously, and satisfies that $\|T\otimes S\| = \|T\|\,\|S\|$. It is well known that if $T$ and $S$ are bounded projections, then $T\otimes S$ is also a bounded projection. In particular, if $Z\subseteq X$ is a $1$-complemented subspace, then $Z\pten Y$ is naturally a $1$-complemented subspace of $X\pten Y$ (see, for instance, \cite[Proposition~2.4]{Ryan}).

\begin{lemma} \label{lemma2:INA}
Let $X, Y$ be Banach spaces. Let $Z$ be a $1$-complemented subspace of $Y$. Then
\[
(X\pten Z)\cap \INA_{\pi}(X\pten Y)=\INA_{\pi}(X\pten Z).
\]
\end{lemma}

\begin{proof}
We only need to prove the inclusion
\[
(X\pten Z)\cap \INA_{\pi}(X\pten Y)\subseteq \INA_{\pi}(X\pten Z) 
\]
as the other one is immediate. Let $u\in (X\pten Z)\cap \INA_{\pi}(X\pten Y)$. Viewing $u$ as an element of $X\pten Y$, the assumption $u\in \INA_{\pi}(X\pten Y)$ yields the existence of a positive Borel measure $\mu$ on $B_X\times B_Y$ such that
\[
u=\int_{B_X\times B_Y} x\otimes y\, d\mu(x,y)
\qquad\text{and}\qquad
\|\mu\|=\|u\|_{\pi}.
\]
Since $Z$ is $1$-complemented in $Y$, there exists a bounded projection $P\colon Y\to Z$ with $\|P\|=1$. So, consider the operator
\[
Q:=\id_X\otimes P \colon X\pten Y \to X\pten Z,
\]
which is a bounded projection satisfying $\|Q\|=1$. By linearity and continuity of $Q$, together with the properties of the Bochner integral, we obtain
\begin{align*}
u
= Q(u)
&= Q\!\left(\int_{B_X\times B_Y} x\otimes y\, d\mu(x,y)\right) \\
&= \int_{B_X\times B_Y} Q(x\otimes y)\, d\mu(x,y) \\
&= \int_{B_X\times B_Y} x\otimes Py\, d\mu(x,y).
\end{align*}
Note that this is well-defined as $P$ is non-expansive.

Let $\mu'$ denote the pushforward of $\mu$ under the mapping $(x,y)\mapsto (x,Py)$, which defines a positive Borel measure on $B_X\times B_Z$. Then the above identity can be rewritten as
\[
u=\int_{B_X\times B_Z} x\otimes z\, d\mu'(x,z).
\]
Moreover, since $\mu$ is positive, we have that $\|\mu'\|=\|\mu\|=\|u\|_{\pi}$. Consequently, $\mu'$ witnesses that $u\in \INA_{\pi}(X\pten Z)$. This completes the proof.
\end{proof}

As an immediate consequence of Lemma \ref{lemma2:INA}, we have the following counterpart of \cite[Lemma 3.1]{GGR}.

\begin{corollary} \label{cor1:INA} Let $X, Y$ be Banach spaces. Suppose that $Z$ is a 1-complemented subspace of $Y$. If $\INA_{\pi}(X \pten Y) = X \pten Y$, then $\INA_{\pi}(X \pten Z) = X \pten Z$.
\end{corollary}

\subsection{Finitely norm-attaining tensors}
Every integral norm-attaining tensor can be approximated by ``classical'' norm-attaining tensors. In fact, a stronger statement holds: such an approximation can be achieved by \emph{finitely norm-attaining tensors}, that is, tensors $u\in X \pten Y$ that admit a finite optimal representation of the form
$$
u = \sum_{k=1}^n x_k \otimes y_k \qquad\mbox{such that}\qquad \sum_{k=1}^n \norm{x_k}\norm{y_k} = \norm{u}_\pi .
$$
Let us denote by $\FNA_{\pi}(X \pten Y)$ the set of {finitely norm-attaining tensors} in $X \pten Y$.

Clearly all $\FNA_\pi$ tensors have finite rank and are norm-attaining. Both reverse implications fail: there are non-norm-attaining tensors in $L_1\pten \ell_2^2=L_1\otimes \ell_2^2$ \cite[Example 3.12(a)]{DJRR}, and every element of $\ell_1\pten Y$ is norm-attaining including those with infinite rank \cite[Proposition 3.6]{DJRR}. 
On the other hand, we do not know whether the equality
$$
\FNA_\pi(X\pten Y)=\NA_\pi(X\pten Y)\cap (X\otimes Y)
$$
holds in general, outside of the easy cases where $X,Y$ are finite-dimensional \cite[Proposition 3.5]{DJRR} and where $X$ is a finite-dimensional polyhedral space and $Y$ is a dual space \cite[Theorem 4.1]{DGJR}.

\begin{theorem} \label{density1} Let $X$ and $Y$ be Banach spaces. Then, every element of $\INA_{\pi}(X \pten Y)$ can be approximated in norm by elements of $\FNA_{\pi}(X \pten Y)$. In other words,
\[
\INA_\pi (X\pten Y) \subseteq \overline{\FNA_\pi (X\pten Y) }.
\]
\end{theorem}

\begin{proof} Let $u \in \INA_\pi (X \pten Y)$ and $\varepsilon>0$ be given. Let $\mu$ be a positive Borel measure on $B_X\times B_Y$ witnessing $u$ as in \eqref{eqINA}. By assumption, $\varphi:(x,y)\mapsto x\otimes y$ is $\mu$-Bochner integrable and hence $\mu$-essentially separably valued, i.e., there exists a $\mu$-null Borel set $N_1\subset B_X\times B_Y$ such that $\varphi(B_X\times B_Y\setminus N_1)$ is separable. Also, by Proposition \ref{theo:1-INA} there exist a $\mu$-null Borel set $N_2\subset B_X\times B_Y$ and $B \in S_{\mathcal{B} (X,Y)}$ such that $B(x,y)=1$ for all $(x,y) \in B_X \times B_Y\setminus N_2$. Set $N=N_1\cup N_2$; we have $\mu(N)=0$.

We now define a countably valued Borel function $f:B_X\times B_Y\to X\pten Y$ approximating $\varphi$ as follows. Let $(z_n)$ be a dense sequence in $\varphi(B_X\times B_Y\setminus N)$, so we may write $z_n=x_n\otimes y_n$ with $(x_n,y_n)\in B_X\times B_Y\setminus N$. Let $E_0=N$, and iteratively define
$$
E_n = \set{(x,y)\in B_X\times B_Y \,:\, \norm{\varphi(x,y) - z_n}\leq\varepsilon} \setminus (E_0\cup E_1\cup\ldots\cup E_{n-1})
$$
for $n\geq 1$. Note that the sets $(E_n)_{n=0}^\infty$ define a partition of $B_X\times B_Y$ into disjoint Borel sets. Finally, set
\[
f(x,y) = \sum_{n=1}^\infty z_n \, \chi_{E_n} (x,y) ;
\]
that is, $f(x,y)=x_n\otimes y_n$ for $(x,y)\in E_n$, $n\geq 1$ and $f(x,y)=0$ for $(x,y)\in N$. By definition, we have $\norm{\varphi-f}\leq\varepsilon$ $\mu$-a.e. (everywhere except on $N$).

Now let
\[
v = \int_{B_X\times B_Y} f\,d\mu = \sum_{n=1}^\infty \mu (E_n) \, x_n\otimes y_n .
\]
and find $k \in \mathbb{N}$ such that $\sum_{n=k+1}^\infty \mu (E_n) < \varepsilon$. Define $v' \in X \pten Y$ by 
\[
v' = \sum_{n=1}^k \mu(E_n) \, x_n \otimes y_n.
\]
Then 
\[
\|v-v'\|_\pi \leq \sum_{n=k+1}^\infty \mu(E_n) < \varepsilon. 
\]
Moreover, we have $\|v'\|_\pi \leq \sum_{n=1}^{k} \mu(E_n)$. On the other hand, 
\[
\|v'\|_\pi \geq | B(v') | = \left| \sum_{n=1}^k \mu(E_n) B(x_n,y_n) \right| = \sum_{n=1}^k \mu(E_n)
\]
as $(x_n,y_n)\in B_X\times B_Y\setminus N_2$ and hence $B(x_n,y_n)=1$. This shows that $v' \in \FNA_\pi (X \pten Y)$. Finally, since $\|\varphi-f\| \leq \varepsilon$ $\mu$-a.e., we have 
\[
\|u-v\|_\pi = \norm{\int_{B_X\times B_Y}(\varphi-f)\,d\mu} \leq \int_{B_X\times B_Y} \norm{\varphi-f}\,d\mu \leq \varepsilon\norm{\mu} = \varepsilon\norm{u}_\pi.
\]
Consequently, $\|u-v'\|_\pi \leq \|u-v\|_\pi+\|v-v'\|_\pi < \varepsilon\|u\|_\pi + \varepsilon.$
\end{proof}

As an immediate consequence of Theorem \ref{density1} we have the following result which says that a Bishop-Phelps type theorem for $\NA_{\pi}(X \pten Y)$ is equivalent to a Bishop-Phelps type theorem for $\INA_{\pi}(X \pten Y)$.

\begin{corollary} \label{INA-NA-are-dense-equivalence} Let $X$ and $Y$ be Banach spaces. We have that 
\begin{equation*}
  \overline{\NA_{\pi}(X \pten Y)}^{\|\cdot\|_{\pi}} = X \pten Y \iff \overline{\INA_{\pi}(X \pten Y)}^{\|\cdot\|_{\pi}} = X \pten Y.
\end{equation*}
\end{corollary}

\begin{remark} In \cite[Section~5]{DJRR} it is shown that there exist tensors which cannot be approximated by norm-attaining tensors. By Corollary~\ref{INA-NA-are-dense-equivalence}, the same counterexample yields the existence of Banach spaces $X$ and $Y$ such that
\[
\overline{\INA_{\pi}(X\pten Y)} \neq X\pten Y.
\]
In particular, since $X\otimes Y$ is dense in $X\pten Y$, there exist finite-rank tensors which do not belong to $\INA_{\pi}(X\pten Y)$, which is the counterpart of \cite[Proposition 5.5]{DJRR}.
\end{remark}

By a combination of Corollary \ref{INA-NA-are-dense-equivalence} and \cite[Lemma 4.1]{GGR}, we also get the following result.

\begin{corollary} \label{cor2:INA} Let $X, Y$ be Banach spaces. Suppose that $Z$ is a 1-complemented subspace of $Y$. If $\overline{\INA_{\pi}(X \pten Y)} = X \pten Y$, then $\overline{\INA_{\pi}(X \pten Z)} = X \pten Z$.
\end{corollary}

\subsection{Norm-attaining extreme points}

We end this section by recording one possible source of interest in norm-attaining tensors. There is currently no known characterization of the extreme points of the unit ball of a projective tensor product $X\pten Y$: at the time of writing the present manuscript, it is an open problem whether they all must be elementary tensors. This problem has only been solved (in the positive) under additional conditions on $X,Y$ or for stronger notions of extreme points; see \cite{GGMR} for a recent account of the state of the art.

One possible application of norm-attaining tensors to this problem lies in the following fact: every norm-attaining extreme point must be an elementary tensor. The following argument shows this for elements of $\INA_\pi$, although it requires the witnessing measure to be Radon.

\begin{proposition}\label{pr:extreme point}
Let $X,Y$ be Banach spaces, and $u$ be an extreme point of $B_{X\pten Y}$. Suppose that $u\in\INA_\pi(X\pten Y)$ is witnessed by a Radon measure. Then $u$ must be an elementary tensor.
\end{proposition}

We recall that the hypothesis of Radonness is automatically satisfied if $X,Y$ have density less than the first measurable cardinal, in particular if they are separable (see Remark \ref{remark:INA-properties}(c) above).

\begin{proof}
Let $\mu$ be a positive Radon measure on $B_X\times B_Y$ representing $u$ as in \eqref{def:INA}. Since $\mu\neq 0$ is Radon, its support is non-empty \cite[Proposition 7.2.9]{Bogachev}; that is, there exists $(x_0,y_0)\in B_X\times B_Y$ such that $\mu(U)>0$ for every open neighborhood of $(x_0,y_0)$. Given $r>0$, let
$$
U_r=(B_X\cap (x_0+rB_X))\times (B_Y\cap (y_0+rB_Y)).
$$
Then $U_r$ contains a neighborhood of $(x_0,y_0)$ and therefore $\mu(U_r)>0$. Now define
$$
v_r = \int_{U_r} x\otimes y\,d\mu(x,y) \qquad\text{and}\qquad w_r = \int_{(B_X\times B_Y)\setminus U_r} x\otimes y\,d\mu(x,y)
$$
and note that $u=v_r+w_r$. Let $\mu_r$ denote the restriction of $\mu$ to $U_r$ (i.e., $\mu_r(E)=\mu(E\cap U_r)$). Then clearly 
\begin{equation*} 
v_r=\int_{B_X\times B_Y} x\otimes y\,d\mu_r(x,y) \ \ \ \mbox{and} \ \ \ \mu_r\ll\mu.
\end{equation*}
By Corollary \ref{rem:ina_ac} we have $v_r\in\INA_\pi(X\pten Y)$ with $\norm{v_r}_\pi=\norm{\mu_r}=\mu(U_r)$. A similar argument shows that $\norm{w_r}_\pi=1-\mu(U_r)$. We claim that $u=v_r/\norm{v_r}_\pi$. Indeed, if $\mu(U_r)=1$, then this is clear. Otherwise, we have
$$
u = v_r+w_r = \mu(U_r)\frac{v_r}{\norm{v_r}_\pi} + (1-\mu(U_r))\frac{w_r}{\norm{w_r}_\pi}
$$
which is a convex combination of two elements of $B_{X\pten Y}$. Since $u$ is an extreme point, our claim follows. Therefore,
\begin{align*}
\norm{x_0\otimes y_0-u}_\pi &= \norm{x_0\otimes y_0-\frac{1}{\mu(U_r)}\int_{U_r}x\otimes y\,d\mu(x,y)}_\pi \\
&= \frac{1}{\mu(U_r)}\norm{\int_{U_r}(x_0\otimes y_0-x\otimes y)\,d\mu(x,y)}_\pi \\
&\leq \frac{1}{\mu(U_r)}\int_{U_r}\norm{x_0\otimes y_0-x\otimes y}_\pi\,d\mu(x,y) \leq 2r .
\end{align*}
As this holds for every $r>0$, we conclude that $u=x_0\otimes y_0$ as we wanted.
\end{proof}

Proposition \ref{pr:extreme point} holds, in particular, for standard norm-attaining tensors. This can also be proved directly with a simplified argument.

\begin{corollary}
Let $X,Y$ be Banach spaces, and $u$ be an extreme point of $B_{X\pten Y}$. If $u\in\NA_\pi(X\pten Y)$, then $u$ is an elementary tensor.
\end{corollary}

\section{A more general version of INA tensors} \label{section:general-ina}

Throughout Section \ref{section:ina}, we have tacitly assumed that the measure $\mu$ is defined on the Borel sets corresponding to the norm topology of $B_X$ and $B_Y$. More generally, one could attempt to define a broader concept of integral norm-attaining tensors as follows. Let $\tau$ be a Hausdorff topology on $B_X\times B_Y$, and set $\Omega := (B_X\times B_Y,\tau)$. We say that an element $u\in X\pten Y$ is \emph{$\tau$-integral projective norm-attaining} if there exists a finite positive $\tau$-Borel measure $\mu$ on $\Omega$ such that the mapping $\varphi: \Omega \to X\pten Y$ defined by $
\varphi(x,y)=x\otimes y$ is $\mu$-Bochner integrable and satisfies
\begin{equation}\label{eq:general ina}
u=\int_\Omega \varphi(x,y)\, d\mu(x,y)
\qquad\text{and}\qquad
\|\mu\|=\|u\|_\pi.
\end{equation}
In this case, we write $u\in\INA_\tau(X\pten Y)$.

The most natural choices for $\tau$ are product topologies $\tau=\tau_X\times\tau_Y$ where $\tau_X$, $\tau_Y$ are locally convex topologies on $B_X$ and $B_Y$, respectively. Besides the norm topology considered in Section \ref{section:ina}, we may use the weak topology, or the weak$^*$ topology in the case of dual spaces. We will write $\INA_w (X\pten Y)$ and $\INA_{w^*} (X^* \pten Y^*)$ when $\Omega$ is respectively given by
\[
(B_X, w)\times (B_Y, w) \qquad\text{and}\qquad (B_{X^*},w^*)\times (B_{Y^*},w^*).
\]
Note that, for topologies $\tau\subseteq\tau'$, every $\tau'$-Borel measure is $\tau$-Borel (more rigorously, its restriction to the Borel $\sigma$-algebra of $\tau$ is), and thus $\INA_{\tau'}\subseteq\INA_\tau$. In particular, we always have $\INA_\pi\subseteq\INA_w\subseteq\INA_{w^*}$.

We already have seen that, for the norm topology, the Bochner integral in \eqref{eq:general ina} is guaranteed to be well-defined under very mild assumptions (see Remark \ref{remark:INA-properties}(c)). As we will see below this is no longer the case when a weaker topology is considered. Nevertheless, most of the results in Section \ref{section:ina} remain valid in this more general framework, with the same argument or possibly only small adjustments. Namely:
\begin{itemize}
\item $\NA_\pi\subseteq\INA_\tau$ for any topology $\tau$, as Dirac deltas are still Borel.
\item $\mu$ is still concentrated on $S_X\times S_Y$ for the weak and weak$^*$ topologies, as our argument only requires that balls are Borel.
\item Considering only positive $\mu$ is no loss of generality if $\tau$ is a product topology.
\item Proposition \ref{theo:1-INA}, Corollary \ref{rem:ina_ac}, Corollary \ref{corollary1:INA} and Proposition \ref{prop1:INA} are valid, with the same proof, for any topology $\tau$.
\item The argument for Lemma \ref{lemma2:INA} and Corollary \ref{cor1:INA} remains valid under the assumption that the projection is continuous with respect to $\tau$. Hence, the $\INA_w$ version is valid in general, but for the $\INA_{w^*}$ version we need to assume that there is a \textit{weak$^*$ continuous} projection $Y\to Z$ with norm $1$.
\item Theorem \ref{density1} and its corollary remain valid for any topology $\tau$. That is, $\FNA_\pi$ is always a norm-dense subset of $\INA_\tau$.
\end{itemize}

\noindent Our proof of Proposition \ref{pr:extreme point}, however, is no longer valid for weaker topologies, as it crucially relies on the existence of bounded open sets.

Ensuring Bochner integrability becomes tricky when $\tau$ is weaker than the norm topology, even when considering Radon measures, because $\varphi$ may fail to be continuous. Perhaps this issue deserves a systematic study, but it will not be pursued here. Instead, we will only record some simple (but rather strong) sufficient conditions for the weak topology. For the next result, recall that a Banach space $X$ has the \emph{Dunford-Pettis property} if $x_n^*(x_n)\to 0$ for every pair of weakly null sequences $(x_n)$, $(x_n^*)$ in $X$ and $X^*$, respectively.

\begin{proposition} \label{tau-result-1}
Let $X$ and $Y$ be separable Banach spaces. Suppose that either
\begin{itemize}
\item[(a)] $X$ or $Y$ has the BAP, or
\item[(b)] $X^*$ and $Y^*$ are separable, and $X$ or $Y$ has the Dunford-Pettis property.
\end{itemize}
Then $\varphi$ is $\mu$-Bochner integrable for any finite positive Borel measure $\mu$ on $(B_X,w)\times (B_Y,w)$.
\end{proposition}

\begin{proof}
Since $\norm{\varphi}$ is bounded and $X,Y$ are separable, we only need to prove that $\varphi$ is weakly measurable. Let $T\in (X\pten Y)^*=\mathcal{L}(X,Y^*)$, we need to check that $T\circ\varphi:(x,y)\to\duality{Tx,y}$ is Borel measurable on $\Omega=(B_X,w)\times (B_Y,w)$.

We start with case (a). First, we note that $T\circ\varphi$ is continuous when $T$ is of finite rank. Indeed, in that case we may write $T=\sum_{k=1}^n x_k^*\otimes y_k^*$ for some $x_k^*\in X^*$, $y_k^*\in Y^*$. Let $(x_i,y_i)$ be a net in $\Omega$ converging to $(x,y)$. Then $x_i\wconv x$, $y_i\wconv y$ and so
$$
\duality{Tx_i,y_i} = \sum_{k=1}^n x_k^*(x_i)y_k^*(y_i) \to \sum_{k=1}^n x_k^*(x)y_k^*(y) = \duality{Tx,y}
$$
by finiteness.

Now let $T$ be arbitrary. Suppose that $X$ has the BAP. Since it is separable, there exists a sequence $(Q_n)$ of finite-rank operators in $\mathcal{L}(X)$ such that $Q_nx\to x$ for all $x\in X$ \cite[Corollary 3.4]{Casazza}. Then $(T\circ Q_n)$ is a sequence of finite rank operators in $\mathcal{L}(X,Y^*)$ such that $TQ_nx\to Tx$ for all $x\in X$, and $T\circ Q_n\circ\varphi$ converges pointwise to $T\circ\varphi$. Since each $T\circ Q_n\circ\varphi$ is continuous, $T\circ\varphi$ is the pointwise limit of a sequence of continuous functions, hence Borel measurable.

In case (b), suppose that $Y$ has the Dunford-Pettis property. We will show that $T\circ\varphi$ is actually continuous. Indeed, $\Omega$ is metrizable because $X^*$, $Y^*$ are separable, so it suffices to check sequential continuity. Let $x_n\wconv x$, $y_n\wconv y$ be weakly converging sequences in $B_X$ and $B_Y$, respectively. Then $Tx_n\wconv Tx$ in $Y^*$, and we have $\duality{Tx_n,y_n} \to \duality{Tx,y}$ by the Dunford-Pettis property.
\end{proof}

Recall that $(X\iten Y)^*= \mathcal{B}_I(X\times Y)$, the space of \emph{integral bilinear forms} on $X\times Y$. These forms and their dual action can be described as follows: given a bilinear mapping $B\in\mathcal{B}(X\times Y)$, its linearization $\tilde{B}$ belongs to $(X\iten Y)^*$ if and only if there is a Radon measure $\mu$ on the compact $\Omega=(B_{X^*}, w^*) \times (B_{Y^*}, w^*)$ such that for every tensor $u\in X\iten Y$, the dual action is given by 
\begin{equation*}
\langle \tilde{B},u \rangle = \int_\Omega u(x^*,y^*) \ d\mu(x^*,y^*),
\end{equation*} 
where $u(x^*, y^*)$ is just the natural action of $u$ as a bilinear form $u\in \mathcal{B}(X^*\times Y^*)$, and $\|\mu\|=\|\tilde{B}\|$ (see \cite[Sections 3.1 and 3.4]{Ryan}). On the other hand, by \cite[Theorem 5.33]{Ryan}, whenever $X^*$ or $Y^*$ has RNP and $X^*$ or $Y^*$ has AP, then $(X\iten Y)^*=X^*\pten Y^*$. We will use these facts in Proposition \ref{prop:separable-reflexive-findim} below without any explicit references.

 \begin{proposition} \label{prop:separable-reflexive-findim}
     Let $X$ and $Y$ be Banach spaces. Suppose that $X^*$ and $Y^*$ are separable. 
     Suppose that $X^*$ has the AP and $Y$ is reflexive. Then, 
 \[
 X^* \pten Y^* = {\INA_{w^*} (X^* \pten Y^*)}.
 \]
 \end{proposition}

\begin{proof}
    Let $z \in X^* \pten Y^*$ be given. As a separable dual space, $X^*$ has the RNP, and by hypothesis, it also has the AP. It follows that $X^* \pten Y^* = (X \iten Y)^*$ (see \cite[Theorem 5.33]{Ryan}). Then there exists a Radon measure $\mu$ on the compact $\Omega = (B_{X^*}, w^*) \times (B_{Y^*}, w^*)$ such that for every $u \in X \iten Y$ we have 
    \begin{equation}\label{eq1_new}
    \langle z, u \rangle = \int_\Omega u(x^*,y^*) \, d\mu(x^*,y^*). 
    \end{equation}
    We claim that the map $\varphi : \Omega \to X^* \pten Y^*$, sending $(x^*,y^*)$ to $x^* \otimes y^*$, is $\mu$-Bochner integrable. To this end, it suffices to prove that $\varphi$ is weakly measurable. 

    Step 1: Let $R \in \mathcal L (X^*, Y) =  (X^* \pten Y^*)^*$ and suppose that $R$ is of finite rank. We will show that $R\circ\varphi$ is Borel measurable. For simplicity, we may assume that $R$ is of rank one, i.e., $R = x_0^{**} \otimes y_0$ for some $x_0^{**} \in X^{**}$ and $y_0 \in Y$.
    Note from the separability of $X^*$ that $(B_{X^{**}}, w^*)$ is a metrizable compact space. This allows us to extract a bounded sequence $(x_n) \subseteq X$ such that $x_n \wstarconv x_0^{**}$, by Goldstine's theorem. For each $n \in \mathbb{N}$, consider $R_n := x_n \otimes y_0 \in \mathcal L (X^*, Y)$. Notice that $R_n \circ \varphi : \Omega \to \mathbb{K}$ is continuous. If $(x_\alpha^*, y_\alpha^*)$ is a net in $\Omega$ such that $x_\alpha^* \wstarconv x_0^*$ and $y_\alpha^* \wstarconv y_0^*$ in $B_{X^*}$ and $B_{Y^*}$, respectively, then 
    \[
    (R_n \circ \varphi) (x_\alpha^*, y_\alpha^*) = y_\alpha^* ( R_n (x_\alpha^*)) = x_\alpha^* (x_n) y_\alpha^* (y_0) \to x_0^* (x_n) y_0^* (y_0) = (R_n \circ \varphi) (x_0^*,y_0^*)
    \]
    as $\alpha \to \infty$. This proves that $R_n \circ \varphi$ is continuous. Observe that $R_n \circ \varphi$ converges to $R \circ \varphi$ pointwise. In fact, for $(x^*, y^*) \in \Omega$, we have 
    \[
    (R_n \circ \varphi) (x^*, y^*) = x^* (x_n) y^* (y_0) \to x_0^{**} (x^*) y^* (y_0) = (R\circ \varphi) (x^*, y^*).
    \]
    In particular, this proves that $R\circ \varphi$ is Borel measurable as the pointwise limit of a sequence of continuous functions.

    Step 2: Next, since $X^*$ is separable and has the AP, it has the BAP \cite[Theorem 3.6]{Casazza}. Thus, for a given $T \in \mathcal{L} (X^*,Y)$, there is a sequence $(T_n) \subseteq \mathcal{L} (X^*, Y)$ of finite-rank operators approximating $T$ strongly. Note that $T_n \circ \varphi$ converges to $T \circ \varphi$ pointwise, and by Step 1 we have that $T_n \circ \varphi$ is Borel measurable for each $n \in \mathbb{N}$. Consequently, we have that $T \circ \varphi$ is Borel measurable as the pointwise limit of a sequence $(T_n \circ \varphi)$ of Borel measurable functions. This proves that $\varphi$ is weakly measurable and completes the proof.
\end{proof}

\begin{remark}
    In Proposition \ref{prop:separable-reflexive-findim}, the separability of $Y$ is assumed only to ensure that $\varphi$ is essentially separably valued. We do not know whether this assumption on $Y$ is essential for the result. 
\end{remark}

Combining Proposition \ref{prop:separable-reflexive-findim} and Corollary \ref{INA-NA-are-dense-equivalence} (which, as we recall, remains valid for $\INA_w$), we immediately obtain the following result. Corollary \ref{prop:seperable-reflexive-finite-dimensional} is covered already by \cite[Corollary 4.6]{DGJR}, although the proof in \cite{DGJR} uses a different approach.

\begin{corollary}\label{prop:seperable-reflexive-finite-dimensional}
Let $X$ and $Y$ be separable reflexive Banach spaces, and suppose that one of them has AP. Then $X \pten Y = \overline{\NA_\pi (X \pten Y)}$. Equivalently, $X \pten Y = \overline{\INA_\pi (X \pten Y)}$.
\end{corollary}

Let us observe that the integral appearing in~\eqref{eq1_new} is scalar-valued, and the proof of Proposition \ref{prop:separable-reflexive-findim} shows that the map 
\[
\varphi : \Omega = (B_{X^*}, w^*) \times (B_{Y^*}, w^*) \to X^* \pten Y^*
\]
is $\mu$-Bochner integrable. Therefore, the following vector-valued integration
\begin{equation} \label{eq1}
z = \int_\Omega x^* \otimes y^* \ d \mu(x^*, y^*).
\end{equation}
is well-defined. 
In general, vector-valued integrals require substantially stronger assumptions than scalar-valued ones, and \eqref{eq1} may fail to be well defined even in classical settings, as illustrated by the following examples. 

\begin{example}\label{example:fails_to_be_Borel}
Suppose that $X=\ell_1$ (so, $X^*$ is not separable) and $Y$ is, say, finite-dimensional. By \cite[Theorem~4]{Christensen}, any element of $\ell_1^{**}$ that is weak$^*$ Borel measurable must actually belong to $\ell_1$. In particular, there exists $\phi\in \ell_1^{**}$ that is not weak$^*$ Borel. Consider the mapping $\varphi: (B_{X^*},w^*)\times (B_{Y^*},w^*) \to X^*\pten Y^*$ defined by 
$\varphi(x^*,y^*)=x^*\otimes y^*$. Fixing $y\in Y$, the scalar-valued function
\[
(x^*,y^*) \longmapsto (\phi\otimes y)\bigl(\varphi(x^*,y^*)\bigr) = \phi(x^*)\cdot y^*(y)
\]
fails to be Borel measurable. Consequently, the mapping $\varphi$ is not weakly measurable and therefore cannot be Bochner integrable. 
\end{example}

The example above shows that the assumptions that $X^*$ or $Y^*$ has the RNP and that one of them has the AP do not guarantee the existence of the Bochner representation \eqref{eq1} (while these assumptions ensure the dual action as in \eqref{eq1_new}). This is also shown by the next example, which we learned from Andreas Defant. We are grateful to him for allowing us to include it in this paper.

\begin{example}\label{example:Defant}
    Let $X=\ell_2$. As $X$ is reflexive and has the BAP, 
\[
(X \widehat{\otimes}_\varepsilon Y)^*
= X^* \widehat{\otimes}_\pi Y^* 
\]
for every Banach space $Y$. Let $\Gamma$ be an uncountable index set and let
$K=\{0,1\}^{\Gamma}$ endowed with the product topology. Denote by
$\lambda$ the product  probability measure on $K$, and set
$Y=C(K)$. Then $Y^*=\mathcal{M}(K)$, the space of regular Borel measures on $K$.
The map
\[
K \to \mathcal{M}(K) \ \ \ \mbox{defined by} \ \ \  t \mapsto \delta_t 
\]
is weak$^*$-measurable and bounded, but its range
$\{\delta_t : t\in K\}$ is norm-discrete and non-separable. Consequently,
this map is not $\lambda$-measurable and therefore cannot be Bochner integrable.

Fix $x_0^* \in B_{X^*}$ with $\|x_0^*\|=1$. Consider the map 
\[
\psi : K \to \Omega:= (B_{X^*}, w^*) \times (B_{Y^*}, w^*) 
\]
given by $\psi (t) =  (x_0^*,\delta_t)$ for every $t \in K$. 
Note that $\psi$ is continuous, in particular, it is a Borel map. Let $\mu = \psi_\sharp \lambda$ the pushforward measure on $\Omega$. 
Then $\mu$ represents a functional $z\in (X \widehat{\otimes}_\varepsilon Y)^*$
via
\[
\langle u,z\rangle
= \int_K u(x_0^*,\delta_t)\, d\lambda(t),
\qquad u\in X \widehat{\otimes}_\varepsilon Y,
\]
and under the above identification we have
$z=x_0^* \otimes \lambda \in X^* \widehat{\otimes}_\pi Y^*$.
However, denoting by $\varphi : \Omega \to X^* \pten Y^*$ the map that sends $(x^*,y^*)$ to $x^* \otimes y^*$, the following representation
\[
z = \int_\Omega \varphi (x^*,y^* ) \, d\mu(x^*, y^*)
= \int_K (\varphi \circ \psi ) (t) \, d\lambda(t) = \int_K x_0^* \otimes \delta_t \, d\lambda(t)
\]
cannot exist as a Bochner integral in
$X^* \widehat{\otimes}_\pi Y^*$, since it would require the Bochner
integrability of the map $t\mapsto \delta_t$ in $\mathcal{M}(K)$, which fails.
\end{example}

Examples \ref{example:fails_to_be_Borel} and \ref{example:Defant} show that the vector-valued integral in~\eqref{eq1} may fail to exist even when the corresponding scalar-valued integrals are well defined, and this highlights the necessity of imposing extra assumptions in order to define $\tau$-integral projective norm-attaining tensors.

\section{On the existence of non-norm-attaining tensors} \label{section:non-norm-attaining}

In this section, we will extend the existing results on norm-attaining tensors to integral norm-attaining tensors. Recall that, by Corollary \ref{INA-NA-are-dense-equivalence}, the problem of density of norm-attaining tensors in $X\pten Y$ is equivalent in the classical and integral settings. Moreover, if every element of $X\pten Y$ is norm-attaining then it is also $\INA_\pi$. Thus, we will focus on the problem of determining pairs $X,Y$ such that $\NA_\pi(X\pten Y)\neq X\pten Y$, as this is the case where replacing $\NA_\pi$ with $\INA_\pi$ provides a formally stronger statement (whether it is really stronger depends on the answer to Question \ref{q:na=ina}).
We will show that all known examples can be extended to the $\INA_\pi$ case. Furthermore, we provide additional examples that are new even in the classical $\NA_\pi$ setting.

We start by reviewing the known examples where $\NA_\pi(X\pten Y)\neq X\pten Y$:
\begin{itemize}
\item[(a)] $X=L_1(\mathbb{T})$ and $Y=\ell_2^2$ \cite[Example 3.12(a)]{DJRR}.
\item[(b)] $X=C(\mathbb{T})^*$ and $Y=\ell_2^2$ \cite[Example 3.8]{GGR}.
\item[(c)] $X=c_0$ (or, more generally, an LFC space, see the text before Corollary \ref{theo:ellps}) and $Y$ a Hilbert space \cite[Theorem 3.1]{Rueda}.
\item[(d)] Any pair of spaces for which $\NA(X,Y^*)$ is not norm-dense in $\mathcal{L}(X,Y^*)$, by way of \cite[Corollary 3.11]{DJRR}. For instance, $L_1\pten L_1$, $L_1\pten Y$ where $Y^*$ is strictly convex and fails RNP, or $G\pten\ell_p$ where $G$ is the Gowers space and $1<p<\infty$ (see \cite[Example 3.12]{DJRR}).
\end{itemize}

For all examples pertaining to case (d), we immediately get $\INA_\pi(X\pten Y)\neq X\pten Y$ by Proposition \ref{prop1:INA}. The extension of case (c) will be covered by Corollary \ref{theo:ellps} below. Cases (a) and (b) are also easy to extend as we can see below.

\begin{example} We have that
\begin{enumerate}
\item $\INA_{\pi}(L_1(\mathbb{T}) \pten \ell_2^2) \not= L_1(\mathbb{T}) \pten \ell_2^2$.
\item $\INA_{\pi}(C(\mathbb{T})^* \pten \ell_2^2) \not= C(\mathbb{T})^* \pten \ell_2^2$.
\end{enumerate}
\end{example}

\begin{proof} 
In \cite[Remark 5.7(2)]{Godefroy15}, a bilinear form $T\in\mathcal{B}(L_1(\mathbb{T})\times \ell_2^2)$ was found such that it attains its norm as a functional on $L_1(\mathbb{T})\pten \ell_2^2$, but it does not attain its norm as an operator from $L_1(\mathbb{T})$ to $(\ell_2^2)^*$. By Corollary \ref{corollary1:INA}, we get that $\INA_\pi(L_1(\mathbb{T})\pten \ell_2^2)\neq L_1(\mathbb{T})\pten \ell_2^2$. Also, since $L_1(\mathbb{T})$ is $1$-complemented in $C(\mathbb{T})^*$ (see \cite[Example 3.8]{GGR}), we get that $\INA_\pi(C(\mathbb{T})^*\pten \ell_2^2)\neq C(\mathbb{T})^*\pten \ell_2^2$ by Corollary \ref{cor1:INA}.
\end{proof}

The remainder of this section is devoted to obtaining new examples of projective tensor products in which non-norm-attaining tensors exist. It is divided into two subsections: the first deals with tensor products involving $c_0$-type spaces and provides an improvement of item (c) above, while the second considers $L_1$-spaces and it improves item (a).

\subsection{\texorpdfstring{$c_0$}{c0}-type spaces}

First of all, we present a result that extends \cite[Theorem 3.1]{Rueda} in the {\it real scalar case}, enlarging the class of examples where the norm-attaining elements lie inside the algebraic tensor $X \otimes Y$.

\begin{theorem} \label{thm:c0Y} Let $\K=\R$ and let $Y$ be a Banach space such that both $Y$ and $Y^*$ are strictly convex. Then,
\begin{equation*}
    \INA_{\pi}(c_0 \pten Y)= \NA_{\pi}(c_0 \pten Y)=\FNA_{\pi}(c_0 \pten Y)\subseteq c_0 \otimes Y.
\end{equation*}
In particular, if in addition $Y$ is infinite-dimensional, then $\INA_\pi (c_0 \pten Y) \neq c_0 \pten Y$.
\end{theorem}

For the proof we need the following lemma, that will also be useful later. As we have not found any reference for it, we prove it for the sake of completeness.

\begin{lemma}\label{lemma:bochner-commute-tensor}
Let $X,Y$ be Banach spaces and $(\Omega,\Sigma,\mu)$ be a finite measure space. Suppose that $f:\Omega\to X$ is $\mu$-Bochner integrable, and let $y_0\in Y$. Then $f\otimes y_0:\Omega\to X\pten Y$, mapping $\omega\in\Omega$ to $f(\omega)\otimes y_0$, is $\mu$-Bochner integrable and
$$
\int_\Omega f\otimes y_0\,d\mu = \pare{\int_\Omega f\,d\mu}\otimes y_0 .
$$
\end{lemma}

\begin{proof}
By assumption, $f$ is weakly measurable. Given $T\in (X\pten Y)^*=\mathcal{L}(X,Y^*)$, we have $y_0\circ T\in X^*$ as $y_0\in Y^{**}$, thus $y_0\circ T\circ f = T\circ (f\otimes y_0)$ is measurable. This means that $f\otimes y_0$ is weakly measurable. If $f(\Omega)$ is essentially contained in a separable subspace $X_0\subset X$, then $(f\otimes y_0)(\Omega)$ is essentially contained in the separable space $X_0\otimes\Span\set{y_0}$, hence $f\otimes y_0$ is $\mu$-measurable. Finally, we have
$$
\int_\Omega\norm{f\otimes y_0}\,d\mu = \int_\Omega\norm{f}\norm{y_0}\,d\mu = \norm{y_0}\int_\Omega\norm{f}\,d\mu < \infty
$$
and therefore $f\otimes y_0$ is Bochner integrable if $f$ is. Now put $x_0 = \int_\Omega f\,d\mu$. For any $T\in (X\pten Y)^*=\mathcal{L}(X,Y^*)$, we have
\begin{align*}
T\pare{\int_\Omega f(x)\otimes y_0\,d\mu(x)} &= \int_\Omega T(f(x),y_0)\,d\mu(x) \\
&= \int_\Omega (y_0\circ T)(f(x))\,d\mu(x) \\
&= (y_0\circ T)\pare{\int_\Omega f(x)\,d\mu(x)} = (y_0\circ T)(x_0) = T(x_0\otimes y_0) .
\end{align*}
Thus we must have the desired equality.
\end{proof}

\begin{proof}[Proof of Theorem \ref{thm:c0Y}]
It is clear that 
\begin{equation*} 
\FNA_{\pi}(c_0 \pten Y)\subseteq\NA_{\pi}(c_0 \pten Y)\subseteq\INA_{\pi}(c_0 \pten Y) 
\end{equation*} 
and also $\FNA_{\pi}(c_0 \pten Y)\subseteq c_0 \otimes Y$. This means we just need to prove $\INA_{\pi}(c_0 \pten Y)\subseteq\FNA_{\pi}(c_0 \pten Y)$. Let $u\in \INA_\pi(c_0\pten Y)$ with $\norm{u}_\pi=1$. Then, there exists a positive Borel measure $\mu$ on $B_{c_0}\times B_Y$ with $\norm{\mu}=1$ and $$u=\int_{B_{c_0}\times B_Y} x\otimes y\ d \mu(x,y).$$  By Proposition \ref{theo:1-INA}, there is some $T\in \mathcal{L}(c_0, Y^*)$ with $\norm{T}=1$ and $T(x)(y)=1$ for $\mu$-almost every pair $(x,y)$. In particular, since $Y^*$ is strictly convex, it follows from \cite[Proposition 4]{Lindenstrauss63} that there is some $i_0\in \N$ such that $T(e_i)=0$ for all $i>i_0$, where $e_i$ are the standard basis elements of $c_0$. Taking the minimal $i_0$ satisfying this condition and reordering the coordinates if necessary, we may assume that $T(e_i)\neq 0$ for all $i\leq i_0$.

We claim that, for $(x,y)\in B_{c_0}\times B_Y$ such that $T(x)(y)=1$, we have $|x(i)|=1$ for all $i\leq i_0$. Indeed, suppose otherwise that $|x(i)|<1$ for some $i\leq i_0$. Then we can find $\varepsilon>0$ such that $|x(i)\pm \varepsilon|\leq 1$, therefore $x\pm \varepsilon e_i\in B_{c_0}$ and $T(x)=\frac{1}{2}(T(x+\varepsilon e_i)+T(x-\varepsilon e_i))$. However, this is impossible because $T(e_i)\neq 0$ and $T(x)$ is an extreme point of $B_{Y^*}$, as $\norm{T(x)}\geq\abs{T(x)(y)}=1$ and $Y^{*}$ is strictly convex. Consider then the set 
\begin{equation*} 
Z=\{ z\in B_{c_0} : |z(i)|=1 \text{ for } 1\leq i\leq i_0, \text{ and } z(i)=0 \text{ for } i>i_0\}. 
\end{equation*} 
It is clear that $Z$ has $2^{i_0}$ elements, so we can write $Z=\{ z_1,\ldots, z_{2^{i_0}}\}$. Define also the set 
\begin{equation*} 
W=\{w\in B_{c_0} : w(i)=0 \text{ for }  1\leq i\leq i_0 \}. 
\end{equation*}
Observe that, if $(x,y)\in B_{c_0}\times B_Y$ are such that $T(x)(y)=1$, then there are unique $z_{j_x}\in Z$ and $w_x\in W$ such that $x=z_{j_x}+w_x$. It is straightforward that $T(x)=T(z_{j_x})$, since $T(W)=\{0\}$. Hence, for each $1\leq k\leq 2^{i_0}$ define the closed set 
\begin{equation*} 
A_k=\{ (x,y)\in B_{c_0}\times B_Y \,:\, T(x)(y)=1 \text{ and } j_x =k\}. 
\end{equation*}
Since $T(z_k)(y)=1$ for all $(x,y)\in A_k$ and $Y$ is strictly convex, we conclude that there is some $y_k\in B_Y$ such that $y=y_k$ for all $(x,y)\in A_k$. 
Moreover, observe that the sets $A_k$ are pairwise disjoint and 
\begin{equation*} 
\bigcup_{1\leq k\leq 2^{i_0}} A_k =\{ (x,y)\in B_{c_0}\times B_Y : T(x)(y)=1\},
\end{equation*}
which is equal to $B_{c_0}\times B_Y$, up to a $\mu$-null set. Finally, for each $1\leq k\leq 2^{i_0}$, write
\begin{equation*}
x_k= \int_{A_k}  x \ d\mu(x,y)\in B_{c_0}, 
\end{equation*}
which is clearly a well-defined Bochner integral.
Hence, using Lemma \ref{lemma:bochner-commute-tensor}, it is easy to see that \begin{align*}
    u&=\int_{B_{c_0}\times B_Y} x\otimes y \ d\mu(x,y)=\sum_{k=1}^{2^{i_0}}\int_{A_k}  x\otimes y\ d\mu(x,y)\\&=\sum_{k=1}^{2^{i_0}}\int_{A_k}  x\otimes y_k\ d\mu(x,y)=\sum_{k=1}^{2^{i_0}}\left(\int_{A_k}  x \ d\mu(x,y)\right)\otimes y_k=\sum_{k=1}^{2^{i_0}} x_k\otimes y_k .
\end{align*} 
Furthermore, since $y_k\in B_Y$ and $A_k\subseteq B_{c_0}\times B_Y$ for all $1\leq k\leq 2^{i_0}$, we conclude
\begin{align*}
    \sum_{k=1}^{2^{i_0}} \norm{x_k}\norm{y_k} \leq \sum_{k=1}^{2^{i_0}}\mu(A_k)\leq \norm{\mu}=\norm{u}_\pi.
\end{align*}
This shows that $u\in \FNA_\pi(c_0\pten Y)$.
\end{proof}

Theorem \ref{thm:c0Y} shows, in particular, that 
\begin{equation*} 
\INA_\pi(c_0\pten L_p)\subseteq c_0\otimes L_p \ \ \ \mbox{for} \ \ \ p\in (1,\infty) 
\end{equation*} 
when considering {\it real} scalars. While we are not able to prove this for complex scalars in full generality, our next result will imply that the same holds true for certain values of $p\in (1,2]$ (see Corollary \ref{theo:ellps} below).

In what follows, we present several additional pairs of Banach spaces $X,Y$ such that $\INA(X\pten Y) \subseteq X\otimes Y$ (see Theorem \ref{theo:ellps-2} below), which are not covered by Theorem \ref{thm:c0Y} above. Before proceeding, we introduce some notation. Let $Z$ be a Gâteaux smooth Banach space. We denote by $\rho: Z \to Z^*$ the mapping that assigns to each $z \in Z$ the unique functional $z^* \in Z^*$ satisfying $\langle z^*, z\rangle = \|z^*\|\cdot\|z\| = \|z\|^2$. 
When $\rho$ is restricted to $S_Z$, it becomes a mapping from $S_Z$ into $S_{Z^*}$; this restriction is usually called the \emph{spherical image map} of $S_Z$.

\begin{theorem} \label{theo:ellps-2}
Let $X$ be a Banach space and $Y$ a Banach space whose dual is Gâteaux smooth. Suppose that
\begin{enumerate}
\item every norm-attaining operator in $\mathcal{L}(X,Y^*)$ is of finite rank and
\item given any finite-dimensional subspace $Z$ of $Y^*$, $\dim (\Span (\rho(Z)))<\infty$.
\end{enumerate}
Then, we have that 
\[
\INA_{\pi}(X\pten Y) \subseteq X\otimes Y. 
\]
In particular, whenever both $X$ and $Y$ are infinite-dimensional Banach spaces, then under the previous assumptions, we have that 
\[
\INA_{\pi}(X\pten Y) \neq X\pten Y.
\]
\end{theorem}

We will provide specific examples of pairs of Banach spaces satisfying the hypotheses of Theorem \ref{theo:ellps-2} in Corollary \ref{theo:ellps} below.

\begin{proof} We adapt the argument from the proof of \cite[Theorem 3.1]{Rueda}. Fix $u\in\INA_{\pi}(X\pten Y)$ with $\norm{u}=1$, so there is a positive finite Borel measure $\mu$ on $B_X\times B_Y$ with $\norm{\mu}=1$ such that
$$
u = \int_{S_{X}\times S_{ Y}} x\otimes y\,d\mu(x,y),
$$
recalling that $\mu$ is concentrated on $S_{X}\times S_{Y}$. 

Let $T\in S_{\mathcal{L}(X, Y^*)}$ with $T(u)=1$.
By Theorem \ref{theo:1-INA}, there exists a $\mu$-null Borel set $N \subseteq S_X \times S_Y$ such that we have $T(x)(y)=1$ for every pair $(x,y) \in S_X \times S_Y \setminus N$. If we denote by $\rho : S_{Y^*} \to S_{Y^{**}}$ the spherical image map for $S_{Y^*}$, then we have $\rho(T(x)) = y \in Y$ for every $(x,y) \in S_X \times S_Y \setminus N$, and
$$
u = \int_{ (S_{X}\times S_{Y}) \setminus N} x\otimes \rho(T(x))\,d\mu(x,y).
$$
In particular, observe that
\[
F:= \spann \{ \rho (T(x)) : (x,y)\in S_X \times S_Y \setminus N \} 
\]
is a subspace of $Y$. Since $T\in\NA(X, Y^*)$, assumption (1) implies that $T$ is of finite rank. Let $Z := T(X) \subseteq Y^*$. By assumption (2), the space $\spann (\rho (Z))$ is a finite-dimensional subspace of $Y^{**}$. Since $F\subseteq Y$ and $F \subseteq  \spann \rho (Z)$, we conclude that $F$ is a finite-dimensional subspace of $Y$. 

Fix a normalized base $\set{w_1,\ldots,w_s}$ of $F$ and let $w_1^*,\ldots,w_s^*$ be the corresponding coordinate functionals; extend them to elements of $Y^*$. Then
\begin{align*}
u &= \int_{(S_{X}\times S_{Y})\setminus N} x\otimes\pare{\sum_{i=1}^s w_i^*(\rho(T(x)))w_i} \,d\mu(x,y) \\
&= \sum_{i=1}^s \int_{(S_{X}\times S_{ Y}) \setminus N} w_i^*(\rho(T(x)))\cdot x\otimes w_i \,d\mu(x,y) \\
&= \sum_{i=1}^s \int_{S_{X}\times S_{ Y}} w_i^*(y)\cdot x\otimes w_i \,d\mu(x,y) \\
&= \sum_{i=1}^s \pare{\int_{S_{X}\times S_{ Y}} w_i^*(y)\cdot x \,d\mu(x,y)}\otimes w_i.
\end{align*}

Note that the measure $\mu$ is concentrated in pairs $(x,y)=(x,\rho(T(x)))$ by its definition, and thus the Bochner integral in the last line is valid and yields an element of $X$ by Lemma \ref{lemma:bochner-commute-tensor}, since the mapping $(x,y)\mapsto w_i^*(y)\cdot x$ is continuous in $B_X\times B_Y$ and therefore measurable.
\end{proof}

Condition (1) in Theorem \ref{theo:ellps-2}, i.e., the condition that all norm-attaining operators are of finite rank (note that this condition was already used in the proof of Theorem \ref{thm:c0Y}) is a phenomenon that is known to occur in several cases in norm-attaining theory. For instance, let us denote by $G$ the predual of the Lorentz sequence space $d(\{\frac{1}{n}\},1)$ (sometimes, referred to as the \emph{Gowers space}) consisting of all sequences $(a_n) \in c_0$ such that 
\begin{equation*}
    \lim_{N \rightarrow \infty} \left( \frac{\sum_{i=1}^N a_i^*}{\sum_{i=1}^N \frac{1}{i}} \right) = 0
 \end{equation*}
 where $(a_n^*)$ is the decreasing rearrangement of $(|a_n|)$. We endow this space with the following norm:
 \begin{equation*}
     \|(a_n)\| := \max_{N \in \N} \left( \frac{\sum_{i=1}^N a_i^*}{\sum_{i=1}^N \frac{1}{i}} \right).
 \end{equation*} 
 It is known \cite{Aguirre} that Gowers space $G$ has the property that every norm-attaining operator from $G$ into any strictly convex space is of finite rank.

Moreover, whenever the norm $X$ is locally dependent upon finitely many coordinates and $Y$ is strictly convex, we have the same by \cite[Lemma 2.8]{Martin}. Recall that the norm of a Banach space $X$ is said to \emph{locally depend upon finitely many coordinates} (LFC, for short) if for every $x \in X \setminus \{0\}$, there are $\eps>0$, a subset $\{x_1^*, \ldots, x_N^*\} \subseteq X^*$ and a continuous function $\varphi: \R^N \to \R$ satisfying $\|y\| = \varphi(x_1^*(y), \ldots, x_N^*(y))$ for $y \in X$ such that $\|y-x\| < \eps$. For instance, spaces of the form $c_0 (\Gamma)$, as well as their closed subspaces, are typical examples of LFC Banach spaces.

On the other hand, condition (2) from Theorem \ref{theo:ellps-2} is trivially satisfied by Hilbert spaces (as $\rho$ is the identity map in that case), but we can show that it is also satisfied by $L_p$ spaces for certain other values of $p$.

\begin{lemma}
Let $(\Omega, \Sigma, \mu)$ be a measure space and $n\in\N$. If $E$ is a finite-dimensional subspace of $L_{2n}(\mu)$, then there is a finite-dimensional subspace $F$ of $L_p(\mu)=L_{2n}(\mu)^*$, where $p=\frac{2n}{2n-1}$, such that $\rho(E)\subset F$.
\end{lemma}

\begin{proof}
For $f\in S_{ L_{2n}(\mu)}$, let $\rho(f)$ be the unique element of $S_{ L_p(\mu)}$ such that $\duality{\rho(f),f}=1$ (note that functions that are equal a.e. are identified, and the spaces $L_p(\mu)$ and $L_{2n}(\mu)$ are both strictly convex). Since $\abs{z}^{2n}=z^{n} \overline{z}^{n}$ for $z\in\C$, we get
$$
\rho(f)(x) = f(x)^{n-1} \overline{f(x)}^{n}.
$$
for almost every $x\in\Omega$.

Let $E=\Span\set{v_1,\ldots,v_r}\subset L_{2n}(\mu)$. Then any $f\in E$ can be written as $f=\sum_{i=1}^r a_iv_i$ for some coefficients $a_i$. Thus, 
\begin{align*}
\rho(f) &= \pare{\sum_{i=1}^r a_iv_i}^{n-1}\pare{\sum_{i=1}^r \overline{a_i}\overline{v_i}}^{n} \\
&= \sum_{\substack{k_1+\ldots+k_r=n-1 \\ j_1+\ldots+j_r=n}} \frac{(n-1)!}{k_1!\ldots k_r!}\frac{n!}{j_1!\ldots j_r!}(a_1v_1)^{k_1}\cdots (a_rv_r)^{k_r}(\overline{a_1}\overline{v_1})^{j_1}\cdots (\overline{a_r}\overline{v_r})^{j_r} \\
&= \sum_{\substack{k_1+\ldots+k_r=n-1 \\ j_1+\ldots+j_r=n}} \frac{(n-1)!\,n!\cdot a_1^{k_1}\ldots a_r^{k_r} \overline{a_1}^{j_1}\cdots \overline{a_r}^{j_r}}{k_1!\cdots k_r!j_1!\cdots j_r!}\cdot v_1^{k_1}\cdots v_r^{k_r}\overline{v_1}^{j_1}\cdots \overline{v_r}^{j_r}
\end{align*}
(where products and powers are assumed to apply pointwise) belongs to the linear span of the finite set
$$
\set{v_1^{k_1}\cdots v_r^{k_r}\overline{v_1}^{j_1}\cdots \overline{v_r}^{j_r} \,:\, k_1+\ldots+k_r=n-1, j_1+\ldots+j_r=n}
$$
which is a subset of $ L_p(\mu)$ by H\"older's inequality.
\end{proof}

\begin{remark}
   The choice of $p=\frac{2n}{2n-1}$ (equivalently, with conjugate exponent $q=2n$) arises naturally in this context. Indeed, the spherical image map $\rho : S_{L_q (\mu)} \to S_{L_p (\mu)}$ is a polynomial in the variables $f$ and $\overline{f}$ if and only if $q$ is an even integer. Moreover, if $q$ is not an even integer, then condition (2) from Theorem \ref{theo:ellps-2} may fail. As a simple example, when $p=\frac{3}{2}$ (so $q=3$), the image of the finite-dimensional subspace $E=\text{span}\{1,x\} \subseteq L_3 [0,1]$ under the map $\rho:L_{3} [0,1]\to L_{\frac{3}{2}} [0,1]$ is infinite-dimensional.
\end{remark}

\begin{corollary} \label{theo:ellps}
Let $(\Omega, \Sigma, \mu)$ be a measure space. Let $n\in\N$ and $p=\frac{2n}{2n-1}$. Then, we have that
\[
\INA_{\pi}(X\pten L_p(\mu)) \subset X\otimes L_p(\mu) 
\]
whenever $X$ is an LFC space (e.g. $c_0$) or Gowers space $G$. 
\end{corollary}

We close this subsection with a remark addressing  simultaneously the density of $\NA_{\pi}$ and of its complement.

\begin{remark}
As observed in \cite[Remark 3.2(3)]{Rueda}, the set $(X\pten Y)\setminus (X\otimes Y)$ is dense in $X\pten Y$ whenever $X$ and $Y$ are infinite-dimensional. Hence, we can enlarge the class of examples of Banach spaces $X, Y$ for which both $\NA_\pi(X\pten Y)$ (resp. $\INA_\pi(X\pten Y)$) and its complement are simultaneously dense. It is clear that this occurs for infinite-dimensional Banach spaces $X$ and $Y$ whenever they satisfy the hypotheses of Theorems \ref{thm:c0Y} or \ref{theo:ellps-2}, and $\overline{\NA_\pi(X\pten Y)} = X\pten Y$ (see the introduction for further comments on the denseness of norm-attaining elements), thus generalizing \cite[Theorem 3.3]{Rueda}. In particular, this is the case for $c_0 \pten L_p$ with $p \in (1,\infty)$ in the real setting, and for $c_0 \pten L_p$ with $p = \frac{2n}{2n-1}$, $n \in \mathbb{N}$, in the complex case. Finally, observe that this phenomenon occurs much more frequently for the set $\FNA_\pi$ than for its analogues $\NA_\pi$ and $\INA_\pi$ as $\FNA_\pi(X\pten Y) \subseteq X \otimes Y$, and therefore its complement, is always dense whenever $X$ and $Y$ are infinite-dimensional.
\end{remark}

\subsection{\texorpdfstring{$L_1$}{L1}-spaces} The problem of norm attainment in $L_1(\mu)\pten Y$ has been treated in recent works. It is known that $\overline{\NA_\pi (L_1(\mu)\pten Y)}=L_1(\mu)\pten Y$  for any measure $\mu$ and any Banach space $Y$ (see for example \cite[Corollary 4.3]{GGR}). Furthermore, when $\mu$ is purely atomic, then every element of $L_1(\mu)\pten Y$ is norm-attaining \cite[Proposition 3.3]{DJRR}. This is also the case for any measure $\mu$ when $Y$ is a finite-dimensional polyhedral space thanks to \cite[Theorem 4.1]{DGJR} and \cite[Lemma 3.1]{GGR} (see \cite[Corollary 3.5]{GGR} for a more general statement), observing that $L_1(\mu)$ is 1-complemented in its bidual. For non-purely-atomic $\mu$, not every element of $L_1(\mu)\pten Y$ is norm-attaining in general; some counterexamples are provided in \cite[Example 3.12]{DJRR}. 

In this section, we isolate a geometric property of $Y$ that is sufficient for non-norm-attaining tensors to exist whenever $\mu$ is not purely atomic.

\begin{theorem} \label{theorem:L1-result2}
    Let $(\Omega, \Sigma, \mu)$ be a measure space with $\mu$ not purely atomic.
    Let $Y$ be a Banach space such that $Y^*$ has the RNP and the following property is satisfied:
    \begin{equation}\label{propertyAGR}
    \begin{split}
    &\text{There exists a Borel measurable map $F:[0,1]\to S_Y$ such that} \\
    &\text{for each $y^*\in S_{Y^*}$ we have $|y^*(F(t))|<1$ for almost all $t\in[0,1]$.}
    \end{split}
    \end{equation}
    Then $$\INA_\pi(L_1(\mu)\pten Y^*)\neq L_1(\mu)\pten Y^*.$$ In particular, $\NA_\pi(L_1(\mu)\pten Y^*)\neq L_1(\mu)\pten Y^*$.
\end{theorem}

Before proving Theorem \ref{theorem:L1-result2}, we present a general class of spaces satisfying condition \eqref{propertyAGR}.

\begin{lemma}\label{lm:AGR sufficient condition}
If $Y$ is a Banach space containing a strictly convex subspace $Z$ with $\dim Z>1$, then $Y$ satisfies \eqref{propertyAGR}.
\end{lemma}

\begin{proof}
     Let $z_1,z_2$ be two linearly independent elements of $S_Z$ and define a continuous curve $F:[0,1]\to S_Z$ by
     $$F(t)=\frac{(1-t)z_1+tz_2}{\norm{(1-t)z_1+tz_2}}$$
     for $t\in [0,1]$; note that this is well-defined as the denominator is non-zero by linear independence.
    Suppose that there are $t_1,t_2\in [0,1]$ and $y^{*}\in S_{Y^{*}}$ such that $|y^{*}(F(t_1))|=1=|y^{*}(F(t_2))|$. Then $\theta_1 y^{*}(F(t_1))=1=\theta_2 y^{*}(F(t_2))$ for some $|\theta_1|=1=|\theta_2|$. Hence, $\|\theta_1 F(t_1)+ \theta_2 F(t_2)\|=2$. This forces $F(t_2)= \frac{\theta_1}{\theta_2} F(t_1)$ by the strict convexity of $Z$. But this implies linear dependence between $z_1$ and $z_2$, and hence a contradiction. Thus, for each $y^*\in S_{Y^*}$ there is at most one value of $t\in [0,1]$ for which $|y^*(F(t))|=1$, and \eqref{propertyAGR} is satisfied.
\end{proof}

    In particular, if $Y$ is strictly convex and $\dim Y>1$, then $Y$ satisfies \eqref{propertyAGR}. However, property \eqref{propertyAGR} is more general than strict convexity: for example, $\ell_1$ satisfies \eqref{propertyAGR} since it contains a two-dimensional strictly convex subspace, see below \cite[Theorem 5]{Lindenstrauss64} (although $Y=\ell_1$ does not satisfy the other hypotheses of Theorem \ref{theorem:L1-result2}). It is also enough if $Y$ contains a two-dimensional subspace $Z$ such that $S_Z$ merely contains a strictly convex part, as we may restrict the construction in Lemma \ref{lm:AGR sufficient condition} to that part of the sphere.

We remark here that we do not know the answer to Question \ref{question:2-dim} below. Notice that if it has a positive answer, then every infinite-dimensional reflexive space $Y$ satisfies the hypotheses in Theorem \ref{theorem:L1-result2}.
    
\begin{question} \label{question:2-dim}
Does every infinite-dimensional reflexive Banach space contain a strictly convex 2-dimensional subspace?
\end{question}

\begin{proof}[Proof of Theorem \ref{theorem:L1-result2}]
Let $\kappa\geq \aleph_0$ be some cardinal. We will first show that the result holds for $L_1(\nu)\pten Y^*$, with $L_1(\nu):= L_1([0,1]^\kappa)$ (i.e. $\nu:= \lambda^\kappa$ where $\lambda$ stands for the Lebesgue measure on $[0,1]$). Recall that, by \eqref{propertyAGR}, there exists a Borel map $F:[0,1]\to S_Y$ such that for all $y^*\in S_{Y^*}$ we have $|y^*(F(t))|<1$ for almost all $t\in [0,1]$. Considering a (continuous) projection $P: [0,1]^\kappa \rightarrow [0,1]$, it follows that $F\circ P : [0,1]^\kappa \rightarrow S_Y$ is  a Borel map and $|y^*(F(P(\omega)))|<1$ for almost all $\omega\in [0,1]^\kappa$. Note that the image of $F$, and hence of $F\circ P$, is separable by virtue of \cite[Lemma 6.10.16]{Bogachev}.

Our argument is inspired in \cite[Remark 5.7(2)]{Godefroy15}. We
 define an operator $T:L_1([0,1]^\kappa)\rightarrow Y\subset Y^{**}$ given by
$$T(h):=\int_{[0,1]^\kappa} h(\omega) F(P(\omega))\ d\nu(\omega),\quad \forall h\in L_1([0,1]^\kappa).$$
This Bochner integral is well defined because the map $F\circ P$ is $\nu$-measurable on $[0,1]^\kappa$ and $\nu$-essentially bounded, while $h$ is $\nu$-measurable and $\nu$-absolutely integrable. Consequently, $T$ takes values in $Y$.

Observe that, for each $h \in L_1 (\nu)$, we have that
\begin{align*}
    \norm{T(h)} \leq \int_{[0,1]^\kappa} \norm{h(\omega)F(P(\omega))} d\nu(\omega) = \int_{[0,1]^\kappa} |h(\omega)| d\nu(\omega)=\norm{h} 
\end{align*} as $\norm{F(P(\omega))}=1$. This proves that $\|T\|\leq 1$.

Given each $h\in L_1(\nu)$ and $y^*\in Y^*$, the evaluation is given by
$$\langle y^*,T(h) \rangle = \int_{[0,1]^\kappa} h(\omega) \langle y^*,F(P(\omega)) \rangle\ d\nu(\omega).$$
By the canonical identification $(L_1 (\nu) \pten Y^*)^* \equiv \call(L_1 (\nu), Y^{**})$, the operator $T$ can also be regarded as a functional on the space $L_1 (\nu) \pten Y^*$, which can be identified with $L_1(\nu, Y^*)$ (see \cite[Example 2.19]{Ryan}). Its dual action for each $\widetilde{g}\in L_1(\nu, Y^*)$ is then given by 
$$\langle T, \, \widetilde{g}\rangle = \int_{[0,1]^\kappa} \langle F(P(\omega)), \, \widetilde{g}(\omega)\rangle\ d\nu(\omega).$$

Since $Y^*$ has the RNP, using \cite[Theorem I.4.2]{DGZ} (which is proved for real Banach spaces but can be easily extended to the complex case) we can find a mapping $f:S_Y\to S_{Y^*}$ such that
\begin{equation}\label{eq:norming-selector}
\langle f(y), y\rangle =1,\quad \|f(y)\|=1,\quad \forall\,y\in S_Y,
\end{equation}
and $f$ is the pointwise limit of a sequence of norm-to-norm continuous mappings.
In particular, $f$ is Borel measurable (see, e.g., \cite[Corollary 6.2.7]{Bogachev}).

Let us show that $T$ attains its norm as a functional. Take $h_0:[0,1]^\kappa\to Y^*$ given by
$$
h_0(\omega):=f(F(P(\omega))),
\quad \omega\in [0,1]^\kappa.
$$
Then it is easy to check that $h_0\in L_1(\nu, Y^*)$ (notice that $h_0$ is $\nu$-essentially separably valued by \cite[Lemma 6.10.16]{Bogachev}, as the image of $F$ is a separable subset of $S_Y$) and that $\langle T, h_0\rangle = \|h_0\| = \|\nu\| = 1$; that is, $T$ attains its norm as a functional on the space $L_1 (\nu) \pten Y^*$ at $h_0$.

We now claim that $T$ does not attain its norm as an operator in $\mathcal{L}(L_1(\nu),Y^{**})$. Note that, if $h_0$ belongs to $\INA_\pi (L_1 (\nu) \pten Y^*)$, then  Corollary \ref{corollary1:INA} shows that $T$ attains its norm as an operator in $\mathcal{L}(L_1(\nu),Y^{**})$. Consequently, our claim implies that $h_0 \notin \INA_\pi (L_1 (\nu) \pten Y^*)$, in particular, 
\[
\INA_\pi (L_1 (\nu) \pten Y^*) \neq L_1 (\nu) \pten Y^*.
\]

To prove the claim, suppose, towards a contradiction, that there exists $g\in S_{L_1(\nu)}$ such that $\norm{T(g)}=1$. Then we can find $y^*\in S_{Y^*}$ with $\langle y^*, T(g)\rangle=1$ (observe that $T(g)\in Y$). 

Now, by the linearity of the Bochner integral, we have
$$
1=\langle y^*, T(g)\rangle =\int_{[0,1]^\kappa}g(\omega) \langle y^*, F(P(\omega))\rangle \ d\nu(\omega) .
$$

Consequently
$$1 \leq \int_{[0,1]^\kappa} |g(\omega)| |\langle y^*, F(P(\omega))\rangle| \ d\nu(\omega) \leq \int_{[0,1]^\kappa} |g(\omega)|\ d\nu(\omega)= \norm{g}=1,$$
Hence equality holds throughout, and therefore
$$\int_{[0,1]^\kappa} |g(\omega)|(1-|\langle y^*, F(P(\omega))\rangle|) \ d\nu(\omega)=0.$$
Since the integrand is nonnegative, it follows that
$$g(\omega)=0 \text{ for $\nu$-a.e. $\omega\in [0,1]^\kappa$ such that $|\langle y^*, F(P(\omega))\rangle|<1$}.$$
On the other hand, we have  $\abs{\duality{y^*,F(P(\omega))}}<1$ $\nu$-a.e., so  $g(\omega)=0$  $\nu$-a.e., which is a contradiction. This completes the proof in the case $\mu=\nu$.

Now, to finish the proof, let $\mu$ be a non-purely atomic measure. Then $L_1(\mu)$ is an $\ell_1$-sum of spaces of the form $L_1(\mu_\gamma)$ where $\mu_\gamma$ is a finite measure for every $\gamma\in\Gamma$ (see \cite[p. 501]{deflo}). Hence, $L_1(\mu)$ contains a $1$-complemented subspace isometric to $L_1(\mu_{\gamma_0})$, where $\mu_{\gamma_0}$ is finite and non-purely atomic. By Maharam's theorem (see, for instance, \cite[Theorem 14.9]{Lacey}), $L_1(\mu_{\gamma_0})$ contains a $1$-complemented subspace isometric to $L_1([0,1]^\kappa)$ for some cardinal $\kappa\geq \aleph_0$. 
That is, $L_1(\mu)$ contains a 1-complemented, isometric copy of $L_1(\nu)$ where $\nu$ is as above. Since we have proved that $\INA_\pi (L_1 (\nu) \pten Y^*) \neq L_1 (\nu) \pten Y^*$, it follows from Corollary \ref{cor1:INA} that $\INA_\pi ( L_1 (\mu) \pten Y^* ) \neq L_1 (\mu) \pten Y^*$. 
\end{proof}

\begin{corollary}
Let $(\Omega, \Sigma, \mu)$ be a measure space with $\mu$ not purely atomic. Suppose that $Y$ is a Banach space that contains a strictly convex subspace of dimension 2 and such that $Y^*$ has the RNP. Then $\INA_\pi(L_1(\mu)\pten Y^*)\neq L_1(\mu)\pten Y^*$.
\end{corollary}

\noindent 
\textbf{Acknowledgements}: Parts of this research were conducted during visits by S. Dantas and J. Guerrero-Viu to the Universitat Politècnica de València in 2025, and they would like to acknowledge the support received there. The authors would like to thank L.~C.~Garc\'{\i}a-Lirola for several fruitful conversations on the topic of this manuscript.
They are also thankful to Andreas Defant for letting us know about Example \ref{example:Defant}.

\noindent
\textbf{Funding information}: R. J. Aliaga, S. Dantas, and \'O. Rold\'an were supported by Grant PID2021-122126NB-C33 funded by MICIU/AEI/10.13039/501100011033 and by ERDF/EU.
S. Dantas and \'O. Rold\'an were also supported by Grant PID2021-122126NB-C31 funded by MICIU/AEI/10.13039/501100011031 and by ERDF/EU.
J. Guerrero-Viu was supported by FPU24/02284 predoctoral grant funded by MCIU; by grant PID2022-137294NB-I00 funded by MCIN/AEI/10.13039/501100011033 and by ``ERDF A way of making Europe''; by grant E48-23R funded by Diputación General de Aragón (DGA).
M. Jung was supported by the research fund of Hanyang University (HY-202500000003346).


\begin{thebibliography}{FaHaMo}

\bibitem {gasp} \textsc{M. D.~Acosta}, \textsc{F.J.~Aguirre, and R.~Pay\'a}, \textit{There is no bilinear Bishop-Phelps theorem}, Isr. J. Math., \textbf{93}, (1996), 221--227.


\bibitem{AGM} \textsc{M. D.~Acosta}, \textsc{D.~Garc\'ia, and M.~Maestre}, \textit{A multilinear Lindenstrauss theorem}, J. Func. Anal. {\bf 235}, (2006), 122--136.

\bibitem{Aguirre} \textsc{F. J. Aguirre}, Norm-attaining operators into strictly convex Banach spaces, J. Math. Anal. Appl. {\bf 222} (1998), no. 2, 431--437.

\bibitem{APS24}
\textsc{R. J. Aliaga, E. Perneck\'a and R. J. Smith},
\textit{Convex integrals of molecules in Lipschitz-free spaces},
J. Funct. Anal. \textbf{287} (2024), 110560.

\bibitem{AFW} \textsc{R.~M.~Aron, C.~Finet, and E.~Werner}, \textit{Some remarks on norm-attaining $n$-linear forms}, Function Spaces (K. Jarosz, ed.), Lecture Notes in Pure and Appl. Math., {\bf 172}, Marcel Dekker, New York, (1995), 19--28.

\bibitem{ADM} \textsc{R.~M.~Aron, D.~Garc\'ia, and M.~Maestre}, \textit{On norm attaining polynomials}, Publ. Res. Inst. Math. Sci., {\bf 39}, (2003), 165--172.




\bibitem{Bogachev}
\textsc{V. I. Bogachev},
\textit{Measure theory},
Springer Berlin-Heidelberg, 2007.

\bibitem{B} \textsc{J.~Bourgain}, \textit{On dentability and the Bishop-Phelps property}, Israel J. Math., {\bf 28}, (1977), 265--271.

\bibitem{Casazza}
\textsc{P. G. Casazza},
\textit{Approximation properties},
in: Handbook of the geometry of Banach spaces, vol. 1, 271--316, North-Holland, Amsterdam, 2001.

\bibitem {CCGMR} \textsc{B.~Cascales}, \textsc{R.~Chiclana}, \textsc{L.~C.~Garc\'ia-Lirola}, \textsc{M.~Mart\'in, and A.~Rueda Zoca}, \textit{On strongly norm attaining Lipschitz maps}, J. Funct. Anal., \textbf{277}, (2019), 1677--1717.




\bibitem{Choi} \textsc{Y.~S.~Choi}, \textit{Norm attaining bilinear forms on $L_1[0,1]$}, J. Math. Anal. Appl., \textbf{211}, (1997), 295--300.

\bibitem{Christensen}
\textsc{J. P. R. Christensen},
\textit{Borel structures and a topological zero-one law},
Math. Scand. \textbf{29} (1971), 245--255.	

\bibitem{DGJR} \textsc{S. Dantas, L. C. García-Lirola, M. Jung, and A. Rueda Zoca}, On norm attainment in (symmetric) tensor products, Quaest. Math. \textbf{46} (2023), no.~2, 393--409.

\bibitem{DJRR} \textsc{S. Dantas, M. Jung, Ó. Roldán and A. Rueda Zoca}, Norm-attaining tensors and nuclear operators. Mediterr. J. Math. \textbf{19} (2022), no.~1, Paper No.~38, 27~pp.

\bibitem{DM} \textsc{S.~Dantas and R.~Medina},
\textit{On holomorphic functions attaining their weighted norms},
Rev. Real Acad. Cienc. Exactas F\'is. Nat. Ser. A--Mat. \textbf{119} (2025), Paper No.~14.

\bibitem {deflo} \textsc{A. Defant and K. Floret}, \textit{Tensor norms and operator ideals}, North-Holland Publishing Co., Amsterdam, 1993.

\bibitem{DGZ} 
\textsc{R. Deville, G. Godefroy, and V. Zizler}, \textit{Smoothness and renormings in Banach spaces},
Pitman Monographs and Surveys in Pure and Applied Mathematics, vol.~64,
Longman Scientific \& Technical, Harlow; copublished in the United States with John Wiley \& Sons, Inc., New York, 1993.

\bibitem{DU} \textsc{J. Diestel and J. J. Uhl Jr.}, \textit{Vector Measures}, American Mathematical Society, Providence, R.I., 1977.

\bibitem{GGM} \textsc{D. García, B. C. Grecu, M. Maestre}, \textit{Geometry in preduals of spaces of 2-homogeneous polynomials on Hilbert spaces}.
Monatsh. Math. \textbf{157} (2009), no. 1, 55–67.

\bibitem{GGMR}
\textsc{L. C. García-Lirola, G. Grelier, G. Mart\'inez-Cervantes and A. Rueda Zoca},
\textit{Extremal structure of projective tensor products},
Results Math. \textbf{78} (2023), Paper No.~196.

\bibitem{GGR} \textsc{L. C. García-Lirola, J. Guerrero-Viu and A. Rueda Zoca}, \textit{Projective tensor products where every element is norm-attaining},
Banach J. Math. Anal. \textbf{19} (2025), Paper No.~19.

\bibitem{Godefroy15} \textsc{G.~Godefroy}, \textit{A survey on Lipschitz-free Banach spaces}, Comment. Math. \textbf{55} (2015), 89--118.

\bibitem{H} \textsc{R.~E.~Huff}, \textit{Dentability and the Radon-Nikod\'ym property}, Duke Math. J., {\bf 41}, (1974) 111--114.



\bibitem {kms} \textsc{V.~Kadets}, \textsc{M.~Mart\'in, and M.~Soloviova}, \textit{Norm-attaining Lipschitz functionals}, Banach J. Math. Anal., \textbf{10}, (2016), 621--637.



\bibitem{Lacey}
\textsc{H. E. Lacey},
\textit{The isometric theory of classical Banach spaces},
Die Grundlehren der mathematischen Wissenschaften \textbf{208},
Springer-Verlag Berlin-Heidelberg, 1974.



\bibitem{Lindenstrauss63} 
\textsc{J.~Lindenstrauss}, \textit{On operators which attain their norm}, Isr. J. Math. \textbf{1} (1963), 139--148.

\bibitem{Lindenstrauss64} \textsc{J. Lindenstrauss}, \textit{On the extension of operators with a finite-dimensional range},
Illinois J. Math. \textbf{8} (1964), 488--499.



\bibitem{Martin} \textsc{M.~Mart\'{i}n}, \textit{The version for compact operators of Lindenstrauss properties A and B}, RACSAM \textbf{110} (2016), 269--284.



\bibitem{Rueda} \textsc{A. Rueda Zoca}, \textit{Several remarks on norm attainment in tensor product spaces},
Mediterr. J. Math. \textbf{20} (2023), Paper No.~208.

\bibitem{Ryan} \textsc{R. A. Ryan}, \textit{Introduction to Tensor Products of Banach Spaces},
Springer Monographs in Mathematics,
Springer, London, 2002.

\bibitem{S} \textsc{W.~Schachermayer}, \textit{Norm attaining operators on some classical Banach spaces}, Pacific J. Math., {\bf 105}, (1983), 427--438.

\bibitem{U} \textsc{J.~J.~Uhl}, \textit{Norm attaining operators on $L_1[0,1]$ and the Radon-Nikod\'ym property}, Pacific J. Math., {\bf 63}, (1976), 293--300.


\end{thebibliography}
\end{document}